\documentclass[11pt]{article}

\usepackage{amsbsy}
\usepackage{amssymb}
\usepackage{amsthm}
\usepackage{amsmath}
\usepackage{fullpage,hyperref,mathrsfs,txfonts,epsfig,graphicx,graphics,verbatim}

\begin{document}

\newcommand{\bwr}{\boldsymbol{\wr}}


\newenvironment{nmath}{\begin{center}\begin{math}}{\end{math}\end{center}}

\newtheorem{thm}{Theorem}[section]
\newtheorem{lem}[thm]{Lemma}
\newtheorem{remark}[thm]{Remark}
\newtheorem{prop}[thm]{Proposition}
\newtheorem{cor}[thm]{Corollary}
\newtheorem{conj}[thm]{Conjecture}
\newtheorem{dfn}[thm]{Definition}
\newtheorem{prob}[thm]{Problem}
\newtheorem{ques}[thm]{Question}


\newcommand{\A}{\mathcal{A}}
\newcommand{\B}{\mathcal{B}}
\newcommand{\K}{\mathcal{K}}
\newcommand{\F}{\mathbb{F}}
\newcommand{\1}{\mathbf{1}}
\newcommand{\s}{\sigma}
\renewcommand{\P}{\mathcal{P}}
\renewcommand{\O}{\Omega}
\renewcommand{\S}{\Sigma}
\renewcommand{\Pr}{\mathbb{P}}
\renewcommand{\approx}{\asymp}
\newcommand{\T}{\mathrm{T}}
\newcommand{\co}{\mathrm{co}}
\newcommand{\Isom}{\mathrm{Isom}}
\newcommand{\edge}{\mathrm{edge}}
\newcommand{\bounded}{\mathrm{bounded}}
\newcommand{\e}{\varepsilon}
\newcommand{\im}{\mathrm{i}}
\newcommand{\restrict}{\upharpoonright}
\newcommand{\supp}{{\mathrm{\bf supp}}}
\renewcommand{\l}{\lambda}
\newcommand{\U}{\mathcal{U}}
\newcommand{\calH}{\mathcal{H}}
\newcommand{\G}{\Gamma}
\newcommand{\g}{\gamma}
\renewcommand{\L}{\Lambda}
\newcommand{\hcf}{\mathrm{hcf}}
\renewcommand{\a}{\alpha}
\newcommand{\N}{\mathbb{N}}
\newcommand{\R}{\mathbb{R}}
\newcommand{\Z}{\mathbb{Z}}
\newcommand{\C}{\mathbb{C}}

\newcommand{\E}{\mathbb{E}}
\newcommand{\alp}{\alpha^*}

\newcommand{\bb}[1]{\mathbb{#1}}
\renewcommand{\rm}[1]{\mathrm{#1}}
\renewcommand{\cal}[1]{\mathcal{#1}}

\newcommand{\fin}{\nolinebreak\hspace{\stretch{1}}$\lhd$}

\title{Embeddings of discrete groups and the speed of random walks}
\author{Assaf Naor\footnote{Research supported in part by NSF grants
CCF-0635078 and DMS-0528387.}\\Courant Institute\\{\tt
naor@cims.nyu.edu} \and Yuval Peres \footnote{Research supported in
part by NSF grant DMS-0605166.}\\Microsoft Research and UC
Berkeley\\{\tt peres@microsoft.com}}
\date{}

\maketitle

\begin{abstract}
Let $G$ be a group generated by a finite set $S$ and equipped with
the associated left-invariant word metric $d_G$. For a Banach space
$X$ let $\alpha^*_X(G)$ (respectively $\alpha^\#_X(G)$) be the
supremum over all $\alpha\ge 0$ such that there exists a Lipschitz
mapping (respectively an equivariant mapping) $f:G\to X$ and  $c>0$
such that for all $x,y\in G$ we have $\|f(x)-f(y)\|\ge c\cdot
d_G(x,y)^\alpha$. In particular, the  {\em Hilbert compression
exponent} (respectively the {\em equivariant Hilbert compression
exponent}) of $G$ is $\alp(G)\coloneqq\alpha^*_{L_2}(G)$
(respectively $\alpha^\#(G)\coloneqq \alpha_{L_2}^\#(G)$). We show
that if $X$ has modulus of smoothness of power type $p$, then $
\alpha^\#_X(G)\le \frac{1}{p\beta^*(G)}$. Here  $\beta^*(G)$ is the
largest $\beta\ge 0$ for which there exists a set of generators $S$
of $G$ and $c>0$ such that for all $t\in \N$ we have
$\E\big[d_G(W_t,e)\big]\ge ct^\beta$, where $\{W_t\}_{t=0}^\infty$
is the canonical simple random walk on the Cayley graph of $G$
determined by $S$, starting at the identity element. This result is
sharp when $X=L_p$, generalizes a theorem of Guentner and
Kaminker~\cite{GK04}, and answers a question posed by
Tessera~\cite{Tess06}. We also show that if $\alpha^*(G)\ge \frac12$
then $ \alpha^*(G\bwr \Z)\ge \frac{2\alpha^*(G)}{2\alpha^*(G)+1}$.
This improves the previous bound due to Stalder and
Valette~\cite{SV07}. We deduce that if we write $\Z_{(1)}\coloneqq
\Z$ and $\Z_{(k+1)}\coloneqq \Z_{(k)}\bwr \Z$ then
$\alpha^*(\Z_{(k)})=\frac{1}{2-2^{1-k}},$ and use this result to
answer a question posed by Tessera in~\cite{Tess06} on the relation
between the Hilbert compression exponent and the isoperimetric
profile of the balls in $G$. We also show that the cyclic
lamplighter groups $C_2\bwr C_n$ embed into $L_1$ with uniformly
bounded distortion, answering a question posed by Lee, Naor and
Peres in~\cite{LeeNaoPer}. Finally, we use these results to show
that edge Markov type need not imply Enflo type.
\end{abstract}

\section{Introduction}

Let $G$ be a finitely generated group\footnote{In this paper all
groups are assumed to be infinite unless stated otherwise.}. Fix a
finite set of generators $S\subseteq G$, which we will always assume
to be symmetric (i.e. $s\in S \iff s^{-1}\in S$). Let $d_G$ be the
left-invariant word metric induced by $S$ on $G$. Given a Banach
space $X$ let $\alpha_X^*(G)$ denote the supremum over all
$\alpha\ge0$ such that there exists a Lipschitz mapping $f:G\to X$
and  $c>0$ such that for all $x,y\in G$ we have $\|f(x)-f(y)\|\ge
c\cdot d_G(x,y)^\alpha$. For $p\ge 1$ we write
$\alpha_{L_p}^*(G)=\alpha_p^*(G)$ and when $p=2$ we write
$\alpha_2^*(G)=\alpha^*(G)$. The parameter $\alpha^*(G)$ is called
the {\em Hilbert compression exponent} of $G$. This quasi-isometric
group invariant was introduced by Guentner and Kaminker
in~\cite{GK04}. We refer to the
papers~\cite{GK04,CN05,AGS06,deCTesVal,Tess06,ADS06,SV07,CSV07} and
the references therein for background on this topic and several
interesting applications.

Analogously to the above definition, one can consider the {\em
equivariant compression exponent} $\alpha^\#_X(G)$, which is defined
exactly as $\alpha^*_X(G)$ with the additional requirement that the
embedding $f:G\to X$ is equivariant (see Section~\ref{sec:equi} for
the definition). As above, we introduce the notation
$\alpha^\#_p(G)=\alpha^\#_{L_p}(G)$ and
$\alpha^\#(G)=\alpha_2^\#(G)$. Clearly $\alpha^\#_X(G)\le
\alpha^*_X(G)$. In the Hilbertian case, when $G$ is amenable we have
$\alpha^*(G)=\alpha^\#(G)$. This was proved by by Aharoni, Maurey
and Mityagin~\cite{AhaMauMit} (see also Chapter 8 in~\cite{BenLin})
when $G$ is Abelian, and by Gromov for general amenable groups
(see~\cite{deCTesVal}).

 The
modulus of uniform smoothness of a Banach space $X$ is defined for
$\tau>0$ as
\begin{eqnarray}\label{def:smoothness}
\rho_X(\tau)=\sup\left\{\frac{\|x+\tau y\|+\|x-\tau y\|}{2}-1:\
x,y\in X,\  \|x\|=\|y\|=1\right\}.
\end{eqnarray}
$X$ is said to be uniformly smooth if $\lim_{\tau\to 0}
\frac{\rho_X(\tau)}{\tau} =0$. Furthermore, $X$ is said to have
modulus of smoothness of power type $p$ if there exists a constant
$K$ such that $\rho_X(\tau)\le K\tau^p$ for all $\tau>0$. It is
straightforward to check that in this case necessarily $p\le 2$. A
deep theorem of Pisier~\cite{Pisier75} states that if $X$ is
uniformly smooth then there exists some $1< p\le 2$ such that $X$
admits an equivalent norm which has modulus of smoothness of power
type $p$. For concreteness we note that $L_p$ has modulus of
smoothness of power type $\min\{p,2\}$. See Section~\ref{sec:equi}
for more information on this topic.

Define $\beta^*(G)$ to be the supremum over all $\beta\ge 0$ for
which there exists a symmetric set of generators $S$ of $G$ and
$c>0$ such that for all $t\in \N$,
\begin{eqnarray}\label{eq:assumption} \E\big[d_G(W_t,e)\big]\ge ct^\beta,
\end{eqnarray}
where here, and in what follows, $\{W_t\}_{t=0}^\infty$ is the
canonical simple random walk on the Cayley graph of $G$ determined
by $S$, starting at the identity element $e$. In~\cite{ANP07}
Austin, Naor and Peres used the method of {\em Markov type} to show
that if $G$ is amenable and $X$ has modulus of smoothness of power
type $p$ then
\begin{eqnarray}\label{eq:ANP}
\alpha^*_X(G)\le \frac{1}{p\beta^*(G)}.
\end{eqnarray}
Our first result, which is proved in Section~\ref{sec:equi},
establishes the same bound as~\eqref{eq:ANP} for the equivariant
compression exponent $\alpha^\#_X(G)$, even when $G$ is not
necessarily amenable.

\begin{thm}\label{thm:equialpha} Let $X$ be a Banach space which
has modulus of smoothness of power type $p$. Then
\begin{eqnarray}\label{eq:sharp}
\alpha^\#_X(G)\le \frac{1}{p\beta^*(G)}.
\end{eqnarray}
\end{thm}
Since when $G$ is amenable $\alpha^*(G)=\alpha^\#(G)$,
Theorem~\ref{thm:equialpha} is a generalization of~\eqref{eq:ANP}
when $X=L_2$.

A theorem of Guentner and Kaminker~\cite{GK04} states that if
$\alpha^\#(G)>\frac12$ then $G$ is amenable. Since for a
non-amenable group $G$ we have $\beta^*(G)=1$
(see~\cite{Kest59,Woess00}), Theorem~\ref{thm:equialpha} implies the
Guentner-Kaminker theorem, while generalizing it to non-Hilbertian
targets (when the target space $X$ is a Hilbert space our method
yields a very simple new proof of the Guentner-Kaminker
theorem---see Remark~\ref{rem:hilbert case}). Note that both known
proofs of the Guentner-Kaminker theorem, namely the original proof
in~\cite{GK04} and the new proof discovered by de Cornulier, Tessera
and Valette in~\cite{deCTesVal}, rely crucially on the fact that $X$
is a Hilbert space. It follows in particular from
Theorem~\ref{thm:equialpha} that for $2\le p<\infty$, if
$\alpha_p^\#(G)>\frac12$ then $G$ is amenable. This is sharp, since
in Section~\ref{sec:equi} we show that for the free group on two
generators $\F_2$, for every $2\le p<\infty$ we have
$\alpha_p^\#(\F_2)=\frac12$. This answers a question posed by
Tessera (see Question 1.6 in~\cite{Tess06}).

Theorem~\ref{thm:equialpha} isolates a geometric property (uniform
smoothness) of the target space $X$ which lies at the heart of the
phenomenon discovered by Guentner and Kaminker. Our proof is a
modification of the martingale method developed by Naor, Peres,
Schramm and Sheffield in~\cite{NPSS06} for estimating the speed of
stationary reversible Markov chains in uniformly smooth Banach
spaces. This method requires several adaptations in the present
setting since the random walk $\{W_t\}_{t=0}^\infty$ is not
stationary---we refer to Section~\ref{sec:equi} for the details.

Given two groups $G$ and $H$, the wreath product $G\bwr H$ is the
group of all pairs $(f,x)$ where $f:H\to G$ has finite support (i.e.
$f(z)= e_G$ for all but finitely many $z\in H$) and $x\in H$,
equipped with the product
$$
(f,x)(g,y)\coloneqq \left(z\mapsto f(z)g(x^{-1}z),xy\right).
$$
If $G$ is generated by the set $S\subseteq G$ and $H$ is generated
by the set $T\subseteq H$ then $G\bwr H$ is generated by the set
$\{(e_{G^H},t):\ t\in T\}\cup \{(\delta_s,e_H):\ s\in S\}$. Unless
otherwise stated we will always assume that $G\bwr H$ is equipped
with the word metric associated with this canonical set of
generators (although in most cases our assertions will be
independent of the choice of generators).

The behavior of the Hilbert compression exponent under wreath
products was investigated in~\cite{AGS06,Tess06,SV07,ANP07}. In
particular, Stalder and Valette proved in~\cite{SV07} that
\begin{eqnarray}\label{eq:SV}
\alpha^*(G\bwr \Z)\ge \frac{\alpha^*(G)}{\alpha^*(G)+1}.
\end{eqnarray}
Here we obtain the following improvement of this bound:
\begin{thm}\label{thm:hilbert} For every finitely generated group we
have,
\begin{eqnarray}\label{eq:first case}
\alpha^*(G)\ge \frac{1}{2}\implies \alpha^*(G\bwr \Z)\ge
\frac{2\alpha^*(G)}{2\alpha^*(G)+1},
\end{eqnarray}
and
\begin{eqnarray}\label{eq:second case}
\alpha^*(G)\le \frac{1}{2}\implies \alpha^*(G\bwr \Z)= \alpha^*(G).
\end{eqnarray}
\end{thm}
We refer to Theorem~\ref{thm:Lp} for an analogous bound for
$\alpha_p(G\bwr \Z)$, as well as a more general estimate for
$\alpha_p(G\bwr H)$. In addition to improving~\eqref{eq:SV}, we will
see below instances in which~\eqref{eq:first case} is actually an
equality. In fact, we conjecture that~\eqref{eq:first case} holds as
an equality for every amenable group $G$.

\`Ershler~\cite{Ersh01} (see also~\cite{Rev03}) proved that
$\beta^*(G\bwr \Z)\ge\frac{1+\beta^*(G)}{2}$. More generally, in
Section~\ref{sec:beta} we show that
\begin{eqnarray}\label{eq:gromov}
\beta^*(G\bwr H)\ge \left\{\begin{array}{ll} \frac{1+\beta^*(G)}{2}
& \text{if $H$ has linear growth,}\\
1 & \text{otherwise.}\end{array}\right.
\end{eqnarray}
Since if $G$ is amenable then $G\bwr \Z$ is also amenable (see
e.g.~\cite{Pat88,KV83}) it follows that for an amenable group $G$,
\begin{eqnarray}\label{eq:beta}
\alpha^*(G\bwr \Z)\le \frac{1}{1+\beta^*(G)}.
\end{eqnarray}

\begin{cor}\label{cor:iterated} If $G$ is amenable and $\alpha^*(G)=\frac{1}{2\beta^*(G)}$ then
$$
\alpha^*(G\bwr \Z)=\frac{1}{2\beta^*(G\bwr
\Z)}=\frac{2\alpha^*(G)}{2\alpha^*(G)+1}.
$$
In particular, if we define iteratively $G_{(1)}\coloneqq G$ and
$G_{(k+1)}\coloneqq G_{(k)}\bwr \Z$, then for all $k\ge 1$,
$$
\alpha^*(G_{(k)})=\frac{2^{k-1}\alpha^*(G)}{(2^k-2)\alpha^*(G)+1}.
$$
\end{cor}
Corollary~\ref{cor:iterated} follows immediately from
Theorem~\ref{thm:hilbert} and the bound~\eqref{eq:beta}. Additional
results along these lines are obtained in Section~\ref{sec:embed};
for example (see Remark~\ref{rem:Z^2}) we deduce that
$\alpha^*\left(\Z\bwr \Z^2\right)=\frac12$.

For $r\in \N$ let $J(r)$ be the smallest constant $J>0$ such that
for every $f:G\to \R$ which vanishes outside the ball
$B(e,r)\coloneqq \left\{x\in G:\  d_G(x,e)\le r\right\}$, we have
$$
\left(\sum_{x\in G}f(x)^2\right)^{1/2}\le J\cdot \left(\sum_{x\in
G}\sum_{s\in S}|f(sx)-f(x)|^2\right)^{1/2}.
$$
Let $\text{a}^*(G)$ be the supremum over all $a\ge 0$ for which
there exists $c>0$ such that for all $r\in \N$ we have $J(r)\ge
cr^a$. Tessera proved in~\cite{Tess06} that $\alpha^*(G)\ge
\text{a}^*(G)$ and asked if it is true that
$\alpha^*(G)=\text{a}^*(G)$ for every amenable group $G$ (see
Question 1.4 in~\cite{Tess06}). Corollary~\ref{cor:iterated} implies
that the answer to this question is negative. Indeed,
Corollary~\ref{cor:iterated} implies that the amenable group
$(\Z\bwr \Z)\bwr\Z$ satisfies
\begin{eqnarray}\label{eq:yet1}
\alpha^*\big((\Z\bwr \Z)\bwr\Z\big)=\frac{4}{7}\quad
\mathrm{yet}\quad \text{a}^*\big((\Z\bwr \Z)\bwr\Z\big)\le \frac12.
\end{eqnarray}
In fact, the ratio $\text{a}^*(G)/\alpha^*(G)$ can be arbitrarily
small, since if we denote $\Z_{(1)}\coloneqq\Z$ and
$\Z_{(k+1)}\coloneqq\Z_{(k)}\bwr \Z$ then for $k\ge 2$,
\begin{eqnarray}\label{eq:yet2}
\alpha^*(\Z_{(k)})=\frac{1}{2-2^{1-k}} \quad \mathrm{yet}\quad
\text{a}^*(\Z_{(k)})\le \frac{1}{k-1}.
\end{eqnarray}
To prove~\eqref{eq:yet2}, and hence also its special
case~\eqref{eq:yet1}, note that the assertion in~\eqref{eq:yet2}
about $\alpha^*(\Z_{(k)})$ is a consequence of
Corollary~\ref{cor:iterated}. To prove the upper bound on
$\text{a}^*(\Z_{(k)})$ in~\eqref{eq:yet2} we note that if $G$ is a
finitely generated group such that the probability of return of the
standard random walk $\{W_t\}_{t=0}^\infty$ satisfies
\begin{eqnarray}\label{eq:return}\Pr[W_t=e]\le \exp\left(-Ct^\gamma\right)\end{eqnarray} for some $C,\gamma\in
(0,1)$ and all $t\in \N$, then
\begin{eqnarray}\label{eq:upper tessera}
\text{a}^*(G)\le \frac{1-\gamma}{2\gamma}.
\end{eqnarray}
This implies~\eqref{eq:yet2} since Pittet and
Saloff-Coste~\cite{PS02} proved that for all $k\ge 2$ there exists
$c,C>0$ such that for $G=\Z_{(k)}$ we have for all $t\ge 1$
\begin{eqnarray}\label{eq:PSL}
\exp\left(-Ct^{\frac{k-1}{k+1}}\left(\log
t\right)^{\frac{2}{k+1}}\right)\le \Pr\left[W_t=e\right]\le
\exp\left(-ct^{\frac{k-1}{k+1}}\left(\log
t\right)^{\frac{2}{k+1}}\right).
\end{eqnarray}

The bound~\eqref{eq:upper tessera} is essentially known. Indeed,
assume that $J(r)\ge cr^a$ for every $r\ge 1$. Following the
notation of Coulhon~\cite{Coul00}, for $v\ge 1$ let $\Lambda(v)$
denote the largest constant $\Lambda\ge 0$ such that for all
$\Omega\subseteq G$ with $|\Omega|\le v$, every $f:G\to \R$ which
vanishes outside $\Omega$ satisfies
$$
\Lambda\cdot \sum_{x\in G}f(x)^2\le \sum_{x\in G}\sum_{s\in
S}|f(sx)-f(x)|^2.
$$
Since for $r\ge 2$ we have $|B(e,r)|\le |S|^r$, it follows
immediately from the definitions that $J(r)^2\le
\frac{1}{\Lambda(|S|^r)}$. Theorem 7.1 in~\cite{Coul00} implies that
there exists a constant $K>0$ such that if $e^{Kt^\gamma}\ge |S|$
then,
\begin{multline*}
t\ge \int_{|S|}^{e^{Kt^{\gamma}}}
\frac{dv}{v\Lambda(v)}=\int_1^{\frac{Kt^\gamma}{\log |S|}}
\frac{\log |S|}{\Lambda(|S|^r)}dr\ge \log|S|
\int_1^{\frac{Kt^\gamma}{\log |S|}} J(r)^2 dr\\\ge c^2\log
|S|\int_1^{\frac{Kt^\gamma}{\log |S|}} r^{2a}dr=\frac{c^2\log
|S|}{(2a+1)}\left(\left(\frac{Kt^\gamma}{\log
|S|}\right)^{2a+1}-1\right).
\end{multline*}
Letting $t\to\infty$ it follows that $(2a+1)\gamma\le 1$,
implying~\eqref{eq:upper tessera}.

\begin{remark} {\em In~\cite{Tess06} Tessera asserted that if the opposite inequality
to~\eqref{eq:return} holds true, i.e. if we have $\Pr[W_t=e]\ge
\exp\left(-Kt^\gamma\right)$ for some $\gamma\in (0,1)$, $K>0$, and
every $t\ge 1$, then $\text{a}^*(G)\ge 1-\gamma$. Unfortunately,
this claim is false in general.\footnote{This remark concerns the version
\url{http://arxiv.org/abs/math/0603138v3} of~\cite{Tess06}; after we informed the author
 of this mistake, it was corrected in later versions of ~\cite{Tess06} .} Indeed, if it were true, then using~\eqref{eq:PSL}
we would deduce that $$\text{a}^*\Big(\Big(\big(\Z\bwr \Z\big)\bwr\
\Z\Big)\bwr \Z\Big)=\text{a}^*(\Z_{(4)})\ge \frac25 \, ,$$ but
from~\eqref{eq:yet2} we know that
 $\text{a}^*(\Z_{(4)})\le \frac13$. On inspection of the proof of
Proposition 7.2 in  ~\cite{Tess06} we
see that the argument given there actually yields the lower bound
$\text{a}^*(G)\ge \frac{1-\gamma}{2}$ (note the squares in the first
equation of the proof of Proposition 7.2 in
[\url{http://arxiv.org/abs/math/0603138v3}]). Thus, the original argument presented in  \cite{Tess06}
to establish the lower bound $\text{a}^*(\Z\bwr \Z)\ge \frac23$ only proves that $\text{a}^*(\Z\bwr \Z)\ge \frac13$.
Nevertheless, the lower bound of $\frac23$, which was used crucially
in~\cite{ANP07}, is correct, as follows from our
Theorem~\ref{thm:hilbert}. After the present paper was posted and
sent to Tessera, he  replaced the original argument in~\cite{Tess06}
for the lower bound $\alpha^*(\Z\bwr \Z)\ge \frac23$ by a correct
argument, along the same lines as our proof of
Theorem~\ref{thm:hilbert}.
 \fin}
\end{remark}

In Section~\ref{sec:embed} we show that the cyclic lamplighter group
$C_2\bwr C_n$ admits a bi-Lipschitz embedding into $L_1$ with
distortion independent of $n$ (here, and in what follows $C_n$
denotes the cyclic group of order $n$). This answers a question
posed in~\cite{LeeNaoPer} and in~\cite{ANV07}. In
Section~\ref{sec:edge} we use the notion of Hilbert space
compression to show that $\Z\bwr \Z$ has edge Markov type $p$ for
any $p<\frac43$, but it does not have Enflo type $p$ for any $p>1$.
We refer to Section~\ref{sec:edge} for the relevant definitions.
This result shows that there is no metric analogue of the well known
Banach space phenomenon ``equal norm Rademacher type $p$ implies
Rademacher $p'$ for every $p'<p$" (see~\cite{T-J89}). Finally, in
Section~\ref{sec:open} we present several open problems that arise
from our work.

\section{Equivariant compression and random walks}\label{sec:equi}

In what follows we will use $\approx$ and $\lesssim$, $\gtrsim$ to
denote, respectively, equality or the corresponding inequality up to
some positive multiplicative constant.

Let $X$ be a Banach space. We denote the group of linear isometric
automorphisms of $X$ by $\Isom(X)$. Fix a homomorphism $\pi:G\to
\Isom(X)$, i.e. an action of $G$ on $X$ by linear isometries. A
function $f:G\to X$ is called a $1$-cocycle with respect to $\pi$ if
for every $x,y\in G$ we have $f(xy)=\pi(x)f(y)+f(x)$. The space of
all $1$-cocycles with respect to $\pi$ is denoted $Z^1(G,\pi)$.
Equivalently, $f\in Z^1(G,\pi)$ if and only if $v\mapsto
\pi(x)v+f(x)$ is an action of $G$ on $X$ by affine isometries. A
function $f:G\to X$ is called a $1$-cocycle if there exists a
homomorphism $\pi:G\to \Isom(X)$ such that $f\in Z^1(G,\pi)$. A
mapping $\psi:G\to X$ is called equivariant if it is given by the
orbit of a vector $v\in X$ under an { affine} isometric action of
$G$ on $X$, or equivalently  $\psi(x)=\pi(x)v+f(x)$ for some
homomorphism $\pi:G\to \Isom(X)$ and $f\in Z^1(G,\pi)$. Note that
since the function $x\mapsto \pi(x)v$ is bounded, the compression
exponents of $\psi$ and $f$ coincide. Therefore in order to bound
the equivariant compression exponent of $G$ in $X$ it suffices to
study the growth rate of $1$-cocycles.

Recall the definition~\eqref{def:smoothness} of the modulus of
uniform smoothness $\rho_X(\tau)$, and that $X$ is said to have
modulus of smoothness of power type $p$ if there exists a constant
$K$ such that $\rho_X(\tau)\le K\tau^p$ for all $\tau>0$. By
Proposition 7 in~\cite{BCL94}, $X$ has modulus of smoothness of
power type $p$ if and only if there exists a constant $S>0$ such
that for every $x,y\in X$
\begin{eqnarray}\label{eq:two point}
\|x+y\|^p+\|x-y\|^p\le 2\,\|x\|^p+2\,S^p\,\|y\|^p.
\end{eqnarray}
The infimum over all $S$ for which \eqref{eq:two point} holds is
called the $p$-smoothness constant of $X$, and is denoted $S_p(X)$.

It was shown in~\cite{BCL94} (see also~\cite{figiel}) that
  $S_2(L_p)\le
\sqrt{p-1}$ for $2\le p<\infty$ and $S_p(L_p)\le 1$ for $1\le p\le
2$ (the order of magnitude of these constants was first calculated
in~\cite{hanner}).

Our proof of Theorem~\ref{thm:equialpha} is based on the following
inequality, which is of independent interest. Its proof is a
modification of the method that was used in~\cite{NPSS06} to study
the Markov type of uniformly smooth Banach spaces.

\begin{thm}\label{thm:equimarkov}
Let $X$ be a Banach space with modulus of smoothness of power type
$p$, and assume that $f:G\to X$ is a $1$-cocycle. Then for every
time $t\in \mathbb N$,
$$
\E \left[\|f(W_t)\|^p\right]\le
C_p(X)t\cdot\E\left[\|f(W_{1})\|^p\right],
$$
where $C_p(X)= \frac{2^{2p}S_p(X)^p}{2^{p-1}-1}$.
\end{thm}
Theorem~\ref{thm:equimarkov} shows that images of
$\{W_t\}_{t=0}^\infty$ under $1$-cocycles satisfy an inequality
similar to the Markov type inequality (note that
$f(W_0)=f(e)=f(e\cdot e)=\pi(e)f(e)+f(e)=2f(e)$, whence $f(e)=0$).
We stress that one cannot apply Markov type directly in this case
because of the lack of stationarity of the Markov chain
$\{f(W_t)\}_{t=0}^\infty$. We overcome this problem by crucially
using the fact that $f$ is a $1$-cocycle.

Before proving Theorem~\ref{thm:equimarkov} we show how it implies
Theorem~\ref{thm:equialpha}.

\begin{proof}[Proof of Theorem~\ref{thm:equialpha}] Observe
that~\eqref{eq:sharp} is trivial if $\alpha^\#_X(G)\le\frac{1}{p}$
(since $\beta^*(G)\le 1$). So, we may assume that
$\alpha^\#(G)>\frac{1}{p}$. Fix $\frac{1}{p}\le \alpha<
\alpha^\#_X(G)$ and $0<\beta<\beta^*(G)$. Then there exists a
$1$-cocycle $f:G\to X$ satisfying
$$
x,y\in G\implies d_G(x,y)^\alpha\lesssim \|f(x)-f(y)\|\lesssim
d_G(x,y).
$$
In addition we know that $\E\left[d_G(W_t,e)\right]\gtrsim t^\beta$.
An application of Theorem~\ref{thm:equimarkov} yields
\begin{eqnarray}\label{eq:upper use new markov}
\E \left[\|f(W_t)\|^p\right]\lesssim
t\E\left[\|f(W_{1})\|^p\right]=t\E\left[\|f(W_{1})-f(e)\|^p\right]\lesssim
t \E \left[d_G(W_{1},e)^p\right]=t.
\end{eqnarray}
On the other hand, since $p\alpha\ge 1$ we may use Jensen's
inequality to deduce that
\begin{eqnarray}\label{eq:lower use new markov}
\E \left[\|f(W_t)\|^p\right]=\E
\left[\|f(W_t)-f(e)\|^p\right]\gtrsim \E
\left[d_G(W_t,e)^{p\alpha}\right]\ge \big(\E\left[
d_G(W_t,e)\right]\big)^{p\alpha}\gtrsim t^{p\alpha\beta}.
\end{eqnarray}
Combining~\eqref{eq:upper use new markov} and~\eqref{eq:lower use
new markov}, and letting $t\to \infty$, implies that
$p\alpha\beta\le 1$, as required.
\end{proof}

\begin{remark}\label{rem:optimal} {\em Theorem~\ref{thm:equialpha} is
optimal for the class of $L_p$ spaces. Indeed let $\F_2$ denote the
free group on two generators. We claim that for every $p\ge 1$,
\begin{eqnarray}\label{eq:GK}
\alpha^\#_p(\F_2)=\max\left\{\frac12,\frac{1}{p}\right\}.
\end{eqnarray}
Observe that since (trivially) $\beta^*(\F_2)=1$,
Theorem~\ref{thm:equialpha} implies that
$\alpha^\#_p(\F_2)\le\max\left\{\frac12,\frac{1}{p}\right\}$. In the
reverse direction Guentner and Kaminker~\cite{GK04} gave a simple
construction of an equivariant mapping $f:\F_2\to L_p$ satisfying
$\|f(x)-f(y)\|_p\ge d_{\F_2}(x,y)^{1/p}$ for all $x,y\in \F_2$. This
implies~\eqref{eq:GK} for $1\le p\le 2$. The case $p\ge 2$ follows
from Lemma~\ref{lem:p>2} below.\fin}
\end{remark}

\begin{lem}\label{lem:p>2} For every finitely generated group $G$ and every $p\ge
1$ we have $\alpha_p^\#(G)\ge \alpha_2^\#(G)$.
\end{lem}

\begin{proof} In what follows we denote the standard orthonormal basis of $\ell_2(\C)$ by $(e_j)_{j=1}^\infty$.
Let $\gamma$ denote the standard Gaussian measure on $\C$. Consider
the countable product $\Omega\coloneqq \C^{\aleph_0}$, equipped with
the product measure $\mu\coloneqq \gamma^{\aleph_0}$. Let $H$ denote
the subspace of $L_2(\Omega,\mu)$ consisting of all linear
functions. Thus, if we consider the coordinate functions
$g_j:\Omega\to \C$ given by $g(z_1,z_2,\ldots)=z_j$ then $H$ is the
space of all functions $h:\Omega\to \C$ of the form
$h=\sum_{j=1}^\infty a_jg_j$, where the sequence
$(a_j)_{j=1}^\infty\subseteq \C$ satisfies $\sum_{j=1}^\infty
|a_j|^2<\infty$, i.e. $(a_j)_{j=1}^\infty\in \ell_2(\C)$. Note that
we are using here the standard probabilistic fact
(see~\cite{Durrett96}) that $\sum_{j=0}^\infty a_jg_j$ converges
almost everywhere, and has the same distribution as
$\left(\sum_{i=1}^\infty |a_i|^2\right)^{1/2} g_1$ (since
$\{g_j\}_{j=1}^\infty$ are i.i.d. standard complex Gaussian random
variables). This fact also implies  that for every unitary operator
$U:\ell_2(\C)\to \ell_2(\C)$,
$$
Uz\coloneqq \left(\sum_{k=1}^\infty \langle Ue_k,e_j\rangle
z_j\right)_{k=1}^\infty\in \Omega,
$$
is well defined for almost $z\in \Omega$, and therefore $U$ can be
thought of as a measure preserving automorphism $U:\Omega\to \Omega$
(we are slightly abusing notation here, but this will not create any
confusion).

Fix a unitary representation $\pi:G\to \Isom\big(\ell_2(\C)\big)$
and a cocycle $f\in Z^1(G,\pi)$ which satisfies
\begin{eqnarray}\label{eq: into gaussian hilbert}
x,y\in G\implies d_G(x,y)^\alpha\lesssim
\|f(x)-f(y)\|_{\ell_2(\C)}\lesssim d_G(x,y).
\end{eqnarray}
For $x\in G$ and $h\in L_p(\Omega,\mu)$ define $\widetilde
\pi(x)h\in L_p(\Omega,\mu)$ by $\widetilde \pi(x)h(z)=h(\pi(x)z)$.
By the above reasoning, since $\pi(x)$ is a measure preserving
automorphism of $(\Omega,\mu)$, $\widetilde \pi(x)$ is a linear
isometry of $L_p(\Omega,\mu)$, and hence $\widetilde \pi :G\to
\Isom\big(L_p(\Omega,\mu)\big)$ is a homomorphism. Note that since
all the elements of $H$ have a Gaussian distribution, all of their
moments are finite. Hence $H\subseteq L_p(\Omega,\pi)$. We can
therefore define $\widetilde f: G\to L_p(\Omega,\mu)$ by $\widetilde
f(x)\coloneqq \sum_{j=1}^\infty \langle f(x),e_j\rangle g_j\in
H\subseteq L_p(\Omega,\mu)$. It is immediate to check that
$\widetilde f\in Z^1\big(G,\widetilde \pi\big)$ and that for every
$x,y\in G$ we have $\left\|\widetilde f(x)-\widetilde
f(y)\right\|_{L_p(\Omega,\mu)}=\|g_1\|_{L_p(\Omega,\mu)}\cdot\|f(x)-f(y)\|_{\ell_2(\C)}$.
Hence $\widetilde f$ satisfies~\eqref{eq: into gaussian hilbert} as
well.
\end{proof}

\begin{remark} {\em Lemma~\ref{lem:p>2} actually establishes the following
fact: there exists a measure space $(\Omega,\mu)$ and a subspace
$H\subseteq \bigcap _{p\ge 1} L_p(\Omega,\mu)$ which is closed in
$L_p(\Omega,\mu)$ for all $1\le p<\infty$ and such that the
$L_p(\Omega,\mu)$ norm restricted to $H$ is proportional to the
$L_2(\Omega,\mu)$ norm. For any group $G$, any unitary
representation $\pi: G\to \Isom(H)$ can be extended to a
homomorphism $\widetilde \pi : G\to \Isom\big(L_p(\Omega,\mu)\big)$.
The space $H$ is widely used in Banach space theory, and is known as
the {\em Gaussian Hilbert space}. The above corollary about the
extension of group actions was previously noted in~\cite{BFGM07}
under the additional restriction that $1<p\notin 2\Z$, as a simple
corollary of an abstract extension theorem due to
Hardin~\cite{Hardin81} (alternatively this is also a corollary of
the classical Plotkin-Rudin theorem~\cite{Plotkin71,Rudin76}).
Lemma~\ref{lem:p>2} shows that no restriction on $p$ is necessary,
while the theorem of Hardin used in~\cite{BFGM07} does require the
above restriction on $p$. The key point here is the use of the
particular subspace $H\subseteq L_p(\Omega,\mu)$ for which unitary
operators have a simple explicit extension to a linear isometric
automorphism of $L_p(\Omega,\mu)$ for any $1\le p<\infty$.\fin}
\end{remark}

We shall now pass to the proof of Theorem~\ref{thm:equimarkov}. We
will use uniform smoothness via the following famous inequality due
to Pisier~\cite{Pisier75} (for the explicit constant below see
Theorem 4.2 in~\cite{NPSS06}).

\begin{thm}[Pisier]\label{thm:martingale} \, Fix $1<p\le 2$ and let
$\{M_k\}_{k=0}^n\subseteq X$ be a martingale in $X$. Then
$$
\E \left[\|M_n-M_0\|^p\right]\le
\frac{S_p(X)^p}{2^{p-1}-1}\cdot\sum_{k=0}^{n-1}\E
\left[\|M_{k+1}-M_k\|^p\right].
$$
\end{thm}

\begin{proof}[Proof of Theorem~\ref{thm:equimarkov}] By assumption $f(x)\in Z^1(G,\pi)$ for some homomorphism
$\pi:G\to \Isom(X)$. Let $\{\sigma_k\}_{k=1}^\infty$ be i.i.d.
random variables uniformly distributed over $S$. Then for $t\ge 1$
$W_t$ has the same distribution as the random product
$\sigma_1\cdots\sigma_t$.

For every $t\ge 1$ the following identity holds true:
\begin{eqnarray}\label{eq:main identity}
2f\left(W_t\right)=\sum_{j=1}^t
\pi\left(W_{j-1}\right)f\left(\sigma_j\right)-\sum_{j=1}^t
\pi\left(W_j\right)f\left(\sigma_{j}^{-1}\right).
\end{eqnarray}
We shall prove~\eqref{eq:main identity} by induction on $t$. Note
that  every $x\in G$ satisfies $0=f(e)=f\left(x^{-1}\cdot
x\right)=\pi(x)^{-1}f(x)+f\left(x^{-1}\right)$, i.e.
$f(x)=-\pi(x)f\left(x^{-1}\right)$. This implies~\eqref{eq:main
identity} when $t=1$. Hence, assuming the validity of~\eqref{eq:main
identity} for $t$ we can use the identity
$2f(xy)=2f(x)+\pi(x)f(y)-\pi(xy)f\left(y^{-1}\right)$ to deduce that
\begin{eqnarray*}
2f\left(W_{t+1}\right)&=&2f(W_t\sigma_{t+1})\\&=& 2f(W_t)+
\pi(W_t)f(\sigma_{t+1})-\pi(W_{t+1})f\left(\sigma_{t+1}^{-1}\right)\\&=&\sum_{j=1}^t
\pi\left(W_{j-1}\right)f\left(\sigma_j\right)-\sum_{j=1}^t
\pi\left(W_j\right)f\left(\sigma_{j}^{-1}\right)+\pi(W_t)f(\sigma_{t+1})-\pi(W_{t+1})f\left(\sigma_{t+1}^{-1}\right)
\\&=&\sum_{j=1}^{t+1}
\pi\left(W_{j-1}\right)f\left(\sigma_j\right)-\sum_{j=1}^{t+1}
\pi\left(W_j\right)f\left(\sigma_{j}^{-1}\right),
\end{eqnarray*}
proving~\eqref{eq:main identity}.

Define
$$
M_t\coloneqq \sum_{j=1}^t
\pi\left(W_{j-1}\right)\left(f\left(\sigma_j\right)-v\right)=\sum_{j=1}^t
\pi\left(\sigma_1\cdots\sigma_{j-1}\right)\left(f\left(\sigma_j\right)-v\right),
$$
and $$ N_t\coloneqq \sum_{j=1}^t
\pi\left(W_t^{-1}W_j\right)\left(f\left(\sigma_{j}^{-1}\right)-v\right)=\sum_{j=1}^t
\pi\left(\sigma_t^{-1}\cdots\sigma_{j+1}^{-1}\right)\left(f\left(\sigma_{j}^{-1}\right)-v\right),
$$
where $v\coloneqq \E \left[f(W_1)\right]\in X$. Note that since $S$
is symmetric, $\sigma_j^{-1}$ has the same distribution as
$\sigma_j$, and therefore $N_t$ has the same distribution as $M_t$.
Moreover, \eqref{eq:main identity} implies that $
2f\left(W_t\right)=M_t-\pi(W_t)N_t-v+\pi(W_t)v$. Since $\pi(W_t)$ is
an isometry, we deduce that
\begin{multline}\label{eq:convexity bound}
2^p\E\left[\|f\left(W_t\right)\|^p\right]\le 4^{p-1}
\E\left[\|M_t\|^p\right]+4^{p-1}\E\left[\|N_t\|^p\right]+2\cdot
4^{p-1} \|v\|^p\\= 2\cdot 4^{p-1} \E\left[\|M_t\|^p\right]+2\cdot
4^{p-1} \left\|\E \left[f(W_1)\right]\right\|^p\le 2\cdot 4^{p-1}
\E\left[\|M_t\|^p\right]+2\cdot 4^{p-1} \E\left[\left\|
f(W_1)\right\|^p\right].
\end{multline}
Note that for every $t\ge 1$,
\begin{multline*} \E\left[M_t\big|\,
\sigma_0,\ldots,\sigma_{t-1}\right]=\E\left[\sum_{j=1}^t
\pi\left(\sigma_0\cdots\sigma_{j-1}\right)\left(f\left(\sigma_j\right)-v\right)\Big|\,
\sigma_0,\ldots,\sigma_{t-1}\right]\\=M_{t-1}+\pi\left(\sigma_0\cdots\sigma_{t-1}\right)
\left(\E\left[f\left(\sigma_j\right)\right]-v\right)=M_{t-1},
\end{multline*}
Hence $\{M_k\}_{k=0}^\infty$ is a martingale with respect to the
filtration induced by $\{\sigma_k\}_{k=0}^\infty$. By
theorem~\ref{thm:martingale},
\begin{multline}\label{eq:use pisier}
\E\left[\|M_t\|^p\right]\le
\frac{S_p(X)^p}{2^{p-1}-1}\cdot\sum_{k=0}^{t-1}\E
\left[\|M_{k+1}-M_k\|^p\right]=\sum_{k=0}^{t-1}\E
\left[\|f(\sigma_k)-v\|^p\right]\\\le
\frac{S_p(X)^p}{2^{p-1}-1}\cdot t2^{p-1}\left(\E\left[\left\|
f(W_1)\right\|^p\right]+\|v\|^p\right)\le
\frac{2^pS_p(X)^p}{2^{p-1}-1}\cdot t\E\left[\left\|
f(W_1)\right\|^p\right].
\end{multline}
Combining~\eqref{eq:convexity bound} and~\eqref{eq:use pisier}
completes the proof of Theorem~\ref{thm:equimarkov}.
\end{proof}

\begin{remark}\label{rem:hilbert case} {\em When the target space $X$ is Hilbert space
one can prove Theorem~\ref{thm:equialpha} via the following simpler
argument. Using the notation in the proof of
Theorem~\ref{thm:equimarkov} we see that for each $t\in \N$ the
random variables $W_t^{-1}=\sigma_t^{-1}\cdots \sigma_1^{-1}$ and
$W_t^{-1}W_{2t}=\sigma_{t+1}\cdots\sigma_{2t}$ are independent and
have the same distribution as $W_t$. Therefore $Y_1\coloneqq
f\left(W_t^{-1}\right)$ and $Y_2\coloneqq
f\left(W_t^{-1}W_{2t}\right)=\pi\left(W_t^{-1}\right)f(W_{2t})+f\left(W_t^{-1}\right)$
are i.i.d., and hence satisfy
\begin{multline*}
\E\left[\left\|f\left(W_{2t}\right)\right\|^2\right]=\E\left[\left\|\pi\left(W_t^{-1}\right)f\left(W_{2t}\right)\right\|^2\right]=\E\left[\|Y_1-Y_2\|^2\right]=
\E\left[\|Y_1\|^2-2\langle
Y_1,Y_2\rangle+\|Y_2\|^2\right]\\=2\E\left[\|f(W_t)\|^2\right]-2\left\|\E\left[f(W_t)\right]\right\|^2\le
2\E\left[\|f(W_t)\|^2\right].
\end{multline*}
By induction it follows that for every $k\in \N$,
$$
\E\left[\left\|f\left(W_{2^k}\right)\right\|^2\right]\le 2^k
\E\left[\|f(W_1)\|^2\right].
$$
This implies Theorem~\ref{thm:equialpha}, and hence also the
Guentner-Kaminker theorem~\cite{GK04}, by arguing exactly as in the
conclusion of the proof of Theorem~\ref{thm:equialpha}.
 }\fin
\end{remark}

\section{The behavior of $L_p$ compression under wreath
products}\label{sec:wreath}

Given two groups $G, H$ let $\mathscr{L}_G(H)$ denote the wreath
product $G\bwr H$ where the set of generators of $G$ is taken to be
$G\setminus\{e\}$ (i.e. any two distinct elements of $G$ are at
distance $1$ from each other). With this definition it is immediate
to check (see for example the proof of Lemma 2.1 in~\cite{ANV07})
that
\begin{eqnarray}\label{eq:the metric}(f,i),(g,j)\in \mathscr{L}_G(\Z)\implies
d_{\mathscr{L}_G(\Z)}\big((f,i),(g,j)\big)\approx
|i-j|+\max\big\{|k|+1:\ f(k)\neq g(k)\big\}.
\end{eqnarray}
The case $G=C_2$ corresponds to the classical lamplighter group on
$H$.

\begin{lem}\label{lem:lamp} For every group $G$ we have
$\alpha^*\big(\mathscr{L}_G(\Z)\big)=1$.
\end{lem}

\begin{proof}  As shown by Tessera in~\cite{Tess06}, $\alpha^*(C_2\bwr
\Z)=1$ (we provide an alternative explicit embedding exhibiting this
fact in Section~\ref{sec:embed} below). Therefore for every
$\alpha\in (0,1)$ there is a mapping $\theta: C_2\bwr \Z\to L_2$
satisfying
\begin{eqnarray}\label{eq:satisfying}
(x,i),(y,j)\in C_2\bwr \Z\implies d_{C_2\bwr
\Z}\big((x,i),(y,j)\big)^{\alpha}\lesssim\|\theta(x,i)-\theta(y,j)\|_2\lesssim
d_{C_2\bwr \Z}\big((x,i),(y,j)\big).
\end{eqnarray}

Let $\{\e_z\}_{z\in G}$ be i.i.d. $\{0,1\}$ valued Bernoulli random
variables, defined on some probability space $(\Omega,\Pr)$. For
every $f:\Z\to G$ define a random mapping $\e_f: \Z\to C_2$ by
$\e_f(k)=\e_{f(k)}$. We now define an embedding
$F:\mathscr{L}_G(\Z)\to L_2(\Omega,L_2)$ by
$$
F(f,i)\coloneqq \theta(\e_f,i).
$$
Fix $(f,i), (g,j)\in \mathscr{L}_G(\Z)$ and let $k_{\max}\in \Z$
satisfy $f(k_{\max})\neq g(k_{\max})$ and $|k_{\max}|=
\max\big\{|k|:\ f(k)\neq g(k)\big\}$. Then
\begin{multline*}
\|F(f,i)-F(g,j)\|_{L_2(\Omega,L_2)}^2=\E\left[\|\theta(\e_f,i)-\theta(\e_g,j)\|_2^2\right]\stackrel{\eqref{eq:satisfying}}{\lesssim}
\E \left[d_{C_2\bwr \Z}\big((\e_f,i),(\e_g,j)\big)^2\right]
\\\stackrel{\eqref{eq:the
metric}}{\approx}\E \left[\left(|i-j|+\max\big\{|k|+1:\
\e_{f(k)}\neq \e_{g(k)}\big\}\right)^2\right]\le
\left[\left(|i-j|+|k_{\max}|+1\right)^2\right]\stackrel{\eqref{eq:the
metric}}{\approx} d_{\mathscr{L}_G(\Z)}\big((f,i),(g,j)\big)^2.
\end{multline*}
In the reverse direction note that since $f(k_{\max})\neq
g(k_{\max})$ with probability $\frac12$ we have
$\e_{f(k_{\max})}\neq \e_{g(k_{\max})}$. Therefore
\begin{multline*}
\|F(f,i)-F(g,j)\|_{L_2(\Omega,L_2)}^2=
\E\left[\|\theta(\e_f,i)-\theta(\e_g,j)\|_2^2\right]\stackrel{\eqref{eq:satisfying}}{\gtrsim}
\E \left[d_{C_2\bwr
\Z}\big((\e_f,i),(\e_g,j)\big)^{2\alpha}\right]\\\stackrel{\eqref{eq:the
metric}}{\approx}\E \left[\left(|i-j|+\max\big\{|k|+1:\
\e_{f(k)}\neq \e_{g(k)}\big\}\right)^{2\alpha}\right]\gtrsim
\left[\left(|i-j|+|k_{\max}|+1\right)^{2\alpha}\right]\stackrel{\eqref{eq:the
metric}}{\approx}
d_{\mathscr{L}_G(\Z)}\big((f,i),(g,j)\big)^{2\alpha}.
\end{multline*}
This completes the proof of Lemma~\ref{lem:lamp}. \end{proof}

\begin{remark}\label{rem:growth} {\em In~\cite{Tess06} Tessera shows that
if $H$ has volume growth of order $d$ then
\begin{eqnarray}\label{eq:d case}
\alpha^*\big(\mathscr{L}_G(H)\big)\ge \frac{1}{d}.
\end{eqnarray}
Note that Tessera makes this assertion for $\mathscr{L}_F(H)$, where
$F$ is finite (see Section 5.1 in~\cite{Tess06}, and specifically
Remark 5.2 there). But, it is immediate from the proof
in~\cite{Tess06} that the constant factors in Tessera's embedding do
not depend on the cardinality of $F$, and therefore~\eqref{eq:d
case} holds in full generality. Observe that~\eqref{eq:d case} is a
generalization of Lemma~\ref{lem:lamp}, but we believe that the
argument in Lemma~\ref{lem:lamp} which reduces the problem to the
case $G=C_2$ is of independent interest.

The case $H=\Z^2$ in~\eqref{eq:d case} can be proved via the
following explicit embedding. For simplicity we describe it when
$G=C_2$. Fix $0<\alpha<\frac12$ and let
$$
\left\{v_{y,r,g}:\ y\in \Z^2, \ r\in \N\cup\{0\},\ g:y+[-r,r]^2\to
\{0,1\}, \ g\not \equiv 0\right\}
$$
be an orthonormal system of vectors in $L_2$. For simplicity we also
write $v_{y,r,0}=0$. define $\psi:C_2\bwr \Z^2\to \R^2\oplus L_2$ by
$$
\psi(f,x)= x\oplus \left(\sum_{y\in \Z^2\setminus
\{x\}}\sum_{r=0}^\infty
\frac{\max\{1-2r/\|x-y\|_\infty,0\}}{\|x-y\|_\infty^{\frac32-2\alpha}}v_{y,r,f\upharpoonright_{y+[-r,r]^2}}\right).
$$
An elementary (though a little tedious) case analysis shows that
$\psi$ is Lipschitz and has compression $\alpha$.
 \fin}
\end{remark}

The following theorem, in combination with Lemma~\ref{lem:lamp},
contains Theorem~\ref{thm:hilbert} as a special case (note
that~\eqref{eq:second case} follows from~\eqref{eq:Lp second} since
clearly $\alpha^*(G\bwr H)\le \alpha^*(G)$).

\begin{thm}\label{thm:Lp}
Let $G, H$ be groups and $p\ge 1$. Then
$$
\min\left\{\alpha_p^*(G),\alpha_p^*\big(\mathscr{L}_G(H)\big)\right\}\ge
\frac{1}{p}\implies \alpha_p^*(G\bwr H)\ge
\frac{p\alpha_p^*(G)\alpha_p^*(\mathscr{L}_G(H))}{p\alpha_p^*(G)+p\alpha_p^*\big(\mathscr{L}_G(H)\big)-1},
$$
and
\begin{eqnarray}\label{eq:Lp second}
\min\left\{\alpha_p^*(G),\alpha_p^*\big(\mathscr{L}_G(H)\big)\right\}\le
\frac{1}{p}\implies \alpha_p^*(G\bwr H)\ge
\min\left\{\alpha_p^*(G),\alpha_p^*\big(\mathscr{L}_G(H)\big)\right\}.
\end{eqnarray}
\end{thm}

\begin{proof} We shall start with some useful preliminary observations. Let
$(X,d_X)$ be a metric space, $p\ge 1$, and let $\Omega$ be a set. We
denote by $\ell_p(\Omega,X)$ the metric space of all finitely
supported functions $f:\Omega\to X$, equipped with the metric
$$
d_{\ell_p(\Omega,X)}(f,g)\coloneqq \left(\sum_{\omega\in \Omega}
d_X\big(f(\omega),g(\omega)\big)^p\right)^{1/p}.
$$
It is immediate to verify that for every $(f,x),(g,y)\in G\bwr H$ we
have
\begin{eqnarray}\label{eq:L1connection}
d_{G\bwr H}\big((f,x),(g,y)\big)\approx
d_{\mathscr{L}_G(H)}\big((f,x),(g,y)\big)+ d_{\ell_1(H,G)}(f,g).
\end{eqnarray}
Indeed, it suffices to verify the
equivalence~\eqref{eq:L1connection} when $(g,y)$ is the identity
element $(e,e)$ of $G\bwr H$. In this case~\eqref{eq:L1connection}
simply says that in order to move from $(e,e)$ to $(f,x)$ one needs
to visit the locations $z\in H$ where $f(z)\neq e$, and in each of
these locations one must move within $G$ from $e$ to the appropriate
group element $f(z)\in G$.

Another basic fact that we will use is that for every
$(f,x),(g,y)\in G\bwr H$,
\begin{eqnarray}\label{eq:support size}
\big|\{z\in H:\ f(z)\neq g(z)\}\big|\le
d_{\mathscr{L}_G(H)}\big((f,x),(g,y)\big).
\end{eqnarray}
Once more, this fact is entirely obvious: in order to move in
$\mathscr{L}_G(H)$ from $(f,x)$ to $(g,y)$ once must visit all the
locations where $f$ and $g$ differ.

We shall now proceed to the proof of Theorem~\ref{thm:Lp}. Fix
$a<\alpha_p^*(G)$ and $b<\alpha_p^*\big(\mathscr{L}_G(H)\big)$. Then
there exists a function $\psi:G\to L_p$  such that
\begin{eqnarray}\label{eq:psi}
u,v\in G\implies d_G(u,v)^{a}\lesssim\|\psi(u)-\psi(v)\|_p\lesssim
d_G(u,v).
\end{eqnarray}
We also know that there exists a function $\phi:\mathscr{L}_G(H)\to
L_p$ which satisfies
\begin{eqnarray}\label{eq:phi}
u,v\in \mathscr{L}_G(H)\implies
d_{\mathscr{L}_G(H)}(u,v)^{b}\lesssim\|\phi(u)-\phi(v)\|_p\lesssim
d_{\mathscr{L}_G(H)}(u,v).
\end{eqnarray}

Define a function $F:G\bwr H\to L_p\oplus \ell_p(H,L_p)$ by
$$
F(f,x)\coloneqq \phi(f,x)\oplus \left(\psi\circ f\right).
$$
Fix $(f,x),(g,y)\in G\bwr H$ and denote $m\coloneqq
d_{\mathscr{L}_G(H)}\big((f,x),(g,y)\big)$ and $n\coloneqq
d_{\ell_1(H,G)}(f,g)$. We know from~\eqref{eq:L1connection} that
$d_{G\bwr H}\big((f,x),(g,y)\big)\approx m+n$. Now,
\begin{multline*}
\|F(f,x)-F(g,y)\|_p=\left(\|\phi(f,x)-\phi(g,y)\|_p^p+\sum_{z\in
H}\|\psi(f(z))-\psi(g(z))\|_p^p\right)^{1/p}
\\\le\|\phi(f,x)-\phi(g,y)\|_p+\sum_{z\in
H}\|\psi(f(z))-\psi(g(z))\|_p\stackrel{\eqref{eq:psi}\wedge\eqref{eq:phi}}{\lesssim}
m+n\approx d_{G\bwr H}\big((f,x),(g,y)\big).
\end{multline*}

In the reverse direction we have the lower bound
\begin{eqnarray}\label{eq:step1}
\|F(f,x)-F(g,y)\|_p\stackrel{\eqref{eq:psi}\wedge\eqref{eq:phi}}{\gtrsim}\left(
m^{bp}+\sum_{z\in H}d_{G}(f(z),g(z))^{ap}\right)^{1/p}.
\end{eqnarray}
If $ap\le 1$ then $\sum_{z\in H}d_{G}(f(z),g(z))^{ap}\ge
\left(\sum_{z\in H}d_{G}(f(z),g(z))\right)^{ap}=n^{ap}$
and~\eqref{eq:step1} implies that
\begin{eqnarray}\label{eq:ap<1}
\|F(f,x)-F(g,y)\|_p\gtrsim \left(m^{bp}+n^{ap}\right)^{1/p}\gtrsim
(m+n)^{\min\{a,b\}}\gtrsim d_{G\bwr
H}\big((f,x),(g,y)\big)^{\min\{a,b\}}.
\end{eqnarray}

Assume that $ap>1$. It follows from~\eqref{eq:support size} that
$\big|\{z\in H:\ f(z)\neq g(z)\}\big|\le m$. Thus, using H\"older's
inequality,  we see that
\begin{eqnarray}\label{eq:holder}
\sum_{z\in H}d_{G}(f(z),g(z))^{ap}\ge
\frac{1}{m^{ap-1}}\left(\sum_{z\in
H}d_{G}(f(z),g(z))\right)^{ap}=\frac{n^{ap}}{m^{ap-1}}.
\end{eqnarray}
Note that $m^{bp}+\frac{n^{ap}}{m^{ap-1}}\ge
n^{\frac{abp^2}{ap+bp-1}}$, which follows by considering the cases
$m\ge n^{\frac{ap}{ap+bp-1}}$ and $m\le n^{\frac{ap}{ap+bp-1}}$
separately. Hence,
\begin{multline}\label{eq:second}
\|F(f,x)-F(g,y)\|_p\stackrel{\eqref{eq:step1}\wedge\eqref{eq:holder}}{\gtrsim}
\left(m^{bp}+\frac{n^{ap}}{m^{ap-1}}\right)^{1/p}\gtrsim
\max\left\{m^b,n^{\frac{abp}{ap+bp-1}}\right\}\\\gtrsim
\left(m+n\right)^{\min\left\{b,\frac{abp}{ap+bp-1}\right\}}\approx
d_{G\bwr
H}\big((f,x),(g,y)\big)^{\min\left\{b,\frac{abp}{ap+bp-1}\right\}}.
\end{multline}
Note that when $ap>1$, if $b\le \frac{abp}{ap+bp-1}$ then $bp\le 1$.
Therefore~\eqref{eq:ap<1} and~\eqref{eq:second} imply
Theorem~\ref{thm:Lp}. \end{proof}

\begin{remark}\label{rem:Z^2} {\em Theorem~\ref{thm:Lp},
in combination with Remark~\ref{rem:growth} and the results of
Section~\ref{sec:beta} below, imply that if $G$ is amenable and $H$
has quadratic growth then
\begin{eqnarray}\label{eq:to see}
\alpha^*(G\bwr H)=\min\left\{\frac12,\alpha^*(G)\right\}.
\end{eqnarray}
Thus, in particular,
$$
\alpha^*\left(C_2\bwr \Z^2\right)=\alpha^*\left(\Z\bwr
\Z^2\right)=\frac12.
$$
To see~\eqref{eq:to see} note that by Theorem~\ref{thm:ershler} in
Section~\ref{sec:beta} we have $\beta^*(G\bwr H)=1$.
Using~\eqref{eq:ANP} we deduce that $\alpha^*(G\bwr H)\le \frac12$,
and the inequality $\alpha^*(G\bwr H)\le \alpha^*(G)$ is obvious.
The reverse inequality in~\eqref{eq:to see} is a corollary of
Theorem~\ref{thm:Lp} and Remark~\ref{rem:growth}.\fin}
\end{remark}

\section{Embedding the lamplighter group into
$L_1$}\label{sec:embed}

In this section we show that the lamplighter group on the $n$-cycle,
$C_2\bwr C_n$, embeds into $L_1$ with distortion independent of $n$.
This implies via a standard limiting argument that also $C_2\bwr \Z$
embeds bi-Lipschitzly into $L_1$. We present two embeddings of
$C_2\bwr C_n$ into $L_1$. Our first embedding is a variant of the
embedding method used in~\cite{ANV07}. In~\cite{ANV07} there is a
detailed explanation of how such embeddings can be discovered by
looking at the irreducible representations of $C_2\bwr C_n$. The
embedding below can be motivated analogously, and we refer the
interested reader to~\cite{ANV07} for the details. Here we just
present the resulting embedding, which is very simple. Our second
embedding is motivated by direct geometric reasoning rather than the
``dual" point of view in~\cite{ANV07}.

In what follows we slightly abuse the notation by considering
elements $(x,i)\in C_2\bwr C_n$ as an index $i\in C_n$ and a subset
$x\subseteq C_n$. For the sake of simplicity we will denote the
metric on $C_2\bwr C_n$ by $\rho$. The metric $d_{C_n}$ will denote
the canonical metric on the $n$-cycle $C_n$. It is easy to check
(see Lemma 2.1 in~\cite{ANV07}) that
\begin{eqnarray}\label{metric on lamplighter}
(x,j),(y,\ell)\in C_2\bwr C_n\implies \rho\big((x,j),(y,\ell)\big)
\approx d_{C_n}(j,k) + \max_{k \in x\triangle y}\,(d_{C_n}(0,k)+1).
\end{eqnarray}

\noindent{\bf First embedding of $C_2\bwr C_n$ into $L_1$.} We
denote by $\alpha: C_n\to C_n$ the  shift $\alpha(j)=j+1$. Let us
write $\cal{I}$ for the family of all arcs (i.e. connected subsets)
of $C_n$ of length $\lfloor n/3 \rfloor$ (of which there are $n$).
We define an embedding $f: C_2 \bwr C_n\to \bigoplus_{I \in
\cal{I}}\bigoplus_{A \subseteq I}\ell_1(C_n)$ by
\begin{eqnarray*}
f(x,j) \coloneqq \bigoplus_{I \in \cal{I}}\bigoplus_{A \subseteq
I}\left(
(-1)^{|A\cap\alpha^k(x)|}\cdot\frac{\1_I(k+j)+n\1_{C_n\setminus
I}(k+j)}{n^22^{n/3}}\right)_{k \in C_n}.
\end{eqnarray*}

It is immediate to check that the metric on $C_2\bwr C_n$ given by
$\|f(x,j)-f(x',j')\|_1$ is $C_2\bwr C_n$-invariant. Therefore it
suffices to show that $\|f(x,j) - f(\emptyset,0)\|_1\approx
\rho\big( (x,j),(\emptyset,0)\big)$ for all $(x,j)\in C_2\bwr C_n$.

Now,
\begin{eqnarray}\label{eq:formula}
\|f(x,j) - f(\emptyset,0)\|_1 &\approx& \sum_{I \in \cal{I}}\sum_{A
\subseteq I}\left(\frac{\big|\{k\in C_n:\
\1_I(k)+\1_I(k+j)=1\}\big|}{n2^{n/3}} + \sum_{\substack{k \in
C_n\\|A\cap \alpha^k(x)|\
\mathrm{odd}}}\frac{\1_I(k)+n\1_{C_n\setminus
I}(k)}{n^22^{n/3}}\right)\nonumber\\ &\approx&
d_{C_n}(0,j)+\frac{1}{n^22^{n/3}}\sum_{I \in \cal{I}}\sum_{k \in
C_n}\big|\{A\subseteq I: |A\cap \alpha^k(x)|\ \mathrm{odd}\} \big|
\cdot\big(\1_I(k)+n\1_{C_n\setminus I}(k)\big)\nonumber\\
&\approx& d_{C_n}(0,j)+\frac{1}{n^2}\sum_{I \in
\cal{I}}\sum_{\substack{k \in C_n\\ I\cap \alpha^k(x)\neq
\emptyset}}\big(\1_I(k)+n\1_{C_n\setminus I}(k)\big).
\end{eqnarray}

It suffices to prove the Lipschitz condition $\|f(x,j) -
f(\emptyset,0)\|_1\lesssim \rho\big( (x,j),(\emptyset,0)\big)$ for
the generators of $C_2\bwr C_n$, i.e. when $(x,j)\in
\big\{(\{0\},0),(\emptyset,1)\big\}$. This follows immediately
from~\eqref{eq:formula} since when $(x,j)= (\emptyset,1)$ then the
second summand in~\eqref{eq:formula} is empty, and therefore $
\|f(\emptyset,1) - f(\emptyset,0)\|_1\approx 1=
\rho\big((\emptyset,1),(\emptyset,0)\big)$, and
$$
\|f(\{0\},0) - f(\emptyset,0)\|_1\approx \frac{1}{n^2}\sum_{I \in
\cal{I}}\sum_{k \in I}\left(\1_I(k)+n\1_{C_n\setminus
I}(k)\right)\approx 1\lesssim \rho\big((\{0\},0),(\emptyset,0)\big).
$$

To prove the lower bound $\|f(x,j) - f(\emptyset,0)\|_1\gtrsim
\rho\big( (x,j),(\emptyset,0)\big)$ suppose that $\ell \in x$ is a
point of $x$ at a maximal distance from $0$ in $C_n$. By considering
only the terms in~\eqref{eq:formula}  for which $\alpha^{k}(\ell)\in
I$ we see that
\begin{multline*}
\|f(x,j) - f(\emptyset,0)\|_1\gtrsim
d_{C_n}(0,j)+\frac{1}{n^2}\sum_{I \in \cal{I}}\sum_{k \in
\alpha^{-\ell}(I)}\big(\1_I(k)+n\1_{C_n\setminus I}(k)\big)\\\approx
d_{C_n}(0,j)+\frac{1}{n^2}\sum_{I \in \cal{I}} \big|I\cap
\alpha^{-\ell}(I)\big|+\frac{1}{n}\sum_{I \in \cal{I}} \big|
\alpha^{-\ell}(I)\setminus I\big| \gtrsim
d_{C_n}(0,j)+\big(1+d_{C_n}(0,\ell)\big)\gtrsim
\rho\left((x,j),(\emptyset,0)\right).
\end{multline*}
This completes the proof that $f$ is bi-Lipschitz with $O(1)$
distortion.\qed

\begin{remark}{\em Fix $s\in (1/2,1)$ and consider the embedding $f: C_2 \bwr C_n\to
\bigoplus_{I \in \cal{I}}\bigoplus_{A \subseteq I}\ell_2(C_n)$ given
by
\begin{eqnarray*}
f(x,j) \coloneqq \bigoplus_{I \in \cal{I}}\bigoplus_{A \subseteq
I}\left( (-1)^{|A\cap\alpha^k(x)|}\cdot\frac{\1_I(k+j)+\sqrt{n}\cdot
\left[d_{C_n}(k+j,I)\right]^{s-\frac12}}{n2^{n/6}}\right)_{k \in
C_n}.
\end{eqnarray*}
Arguing similarly to~\cite{ANV07} (and the above) shows that
$\rho(u,v)^s\lesssim \|f(u)-f(v)\|_2\lesssim \rho(u,v)$ for all
$u,v\in C_2\bwr C_n$, where the implied constants are independent of
$n$. By a standard limiting argument it follows that
$\alpha^*(C_2\bwr \Z)=1$. This fact was first proved by Tessera
in~\cite{Tess06} via a different approach.\fin}
\end{remark}

\noindent{\bf Second embedding of $C_2\bwr C_n$ into $L_1$.} Let
$\cal{J}$ be the set of all arcs in $C_n$. In what follows for $J\in
\cal{J}$ we let $J^\circ$ denote the interior of $J$. Let
$\{v_{J,A}:\ J\in \cal{J},\ A\subseteq J\}$ be disjointly supported
unit vectors in $L_1$. Define $f:C_2\bwr C_n\to \C\oplus L_1$ by
$$
f(x,j)\coloneqq \left(ne^{\frac{2\pi i j}{n}}\right)\oplus
\left(\frac{1}{n}\sum_{J\in \cal{J}} \1_{\{j\notin J^\circ\}}
v_{J,x\cap J}\right).
$$

As before, since the metric on $C_2\bwr C_n$ given by
$\|f(x,j)-f(x',j')\|_1$ is $C_2\bwr C_n$-invariant, it suffices to
show that $\|f(x,j) - f(\emptyset,0)\|_1\approx \rho\big(
(x,j),(\emptyset,0)\big)$ for all $(x,j)\in C_2\bwr C_n$. Now,
\begin{multline}\label{eq:formula1}
\|f(x,j)-f(\emptyset ,0)\|_1\approx d_{C_n}(0,j)+\frac{1}{n}
\sum_{J\in \cal{J}}\left\|\1_{\{j\notin J^\circ\}} v_{J,x\cap
J}-\1_{\{0\notin J^\circ\}}
v_{J,\emptyset}\right\|_1\\=d_{C_n}(0,j)+\frac{1}{n}
\sum_{\substack{J\in \cal{J}\\ x\cap
J=\emptyset}}\left|\1_{\{j\notin J^\circ\}} -\1_{\{0\notin
J^\circ\}} \right|+\frac{1}{n} \sum_{\substack{J\in \cal{J}\\ x\cap
J\neq\emptyset}}\left(\1_{\{j\notin J^\circ\}} +\1_{\{0\notin
J^\circ\}} \right).
\end{multline}
We check the Lipschitz condition for the generators $ (\emptyset,1)$
and $(\{0\},0)$ as follows:
\begin{eqnarray*}\label{eq:generator1}
\|f(\emptyset,1)-f(\emptyset
,0)\|_1\stackrel{\eqref{eq:formula1}}{\approx}
1+\frac{1}{n}\left|\left\{J\in \cal{J}:\ \left|\{0,1\}\cap
J^\circ\right|=1\right\}\right|\approx 1= \rho\big(
(\emptyset,1),(\emptyset,0)\big),
\end{eqnarray*}
and
\begin{eqnarray*}\label{eq:generator1}
\|f(\{0\},0)-f(\emptyset
,0)\|_1\stackrel{\eqref{eq:formula1}}{\approx}
\frac{1}{n}\left|\left\{J\in \cal{J}:\ 0\in J\setminus
J^\circ\right\}\right|\approx 1= \rho\big(
(\{0\},0),(\emptyset,0)\big).
\end{eqnarray*}
Hence $\|f(x,j) - f(\emptyset,0)\|_1\lesssim \rho\big(
(x,j),(\emptyset,0)\big)$ for all $(x,j)\in C_2\bwr C_n$.

To prove the lower bound $\|f(x,j) - f(\emptyset,0)\|_1\gtrsim
\rho\big( (x,j),(\emptyset,0)\big)$ suppose that $\ell \in x$ is a
point of $x$ at a maximal distance from $0$ in $C_n$. Then
\begin{multline}\label{eq:formula2}
\|f(x,j) - f(\emptyset,0)\|_1\stackrel{\eqref{eq:formula1}}{\gtrsim}
d_{C_n}(0,j)+\frac{1}{n} \sum_{\substack{J\in \cal{J}\\
\ell\in J}}\left(\1_{\{j\notin J^\circ\}} +\1_{\{0\notin J^\circ\}}
\right)\approx d_{C_n}(0,j)+\frac{1}{n}\left|\left\{J\in \cal{J}:\
\ell\in J\ \wedge\  \{0,j\}\setminus J^\circ\neq
\emptyset\right\}\right|\\\gtrsim
d_{C_n}(0,j)+\frac{(\ell+1)(n-\ell)}{n}\approx
d_{C_n}(0,j)+d_{C_n}(0,\ell)+1 \approx \rho\big(
(x,j),(\emptyset,0)\big),
\end{multline}
Where in~\eqref{eq:formula2} we used the fact that the intervals
$\big\{[a,b]:\ a\in \{0,\ldots,\ell\},\
b\in\{\ell,\ldots,n-1\}\big\}$ do not contain $0$ in their interior,
but do contain $\ell$.\qed

\begin{remark} {\em A separable metric space embeds with distortion $D$ into $L_p$ if and
only if all its finite subsets do. Therefore our embeddings for
$C_2\bwr C_n$ into $L_1$ imply that $C_2\bwr \Z$ admits a
bi-Lipschitz embedding into $L_1$. This can also be seen via the
explicit embedding $F(x,j)\coloneqq j\oplus
\big(\psi(x,j)-\psi(0,0)\big)$, where
$$
F(x,j)\coloneqq \sum_{k\ge j} v_{[k,\infty),x\cap
[k,\infty)}+\sum_{k\le j} v_{(-\infty,k],x\cap (-\infty,k]},
$$
and $\{v_{J,A}:\ J\in \{[k,\infty)\}_{k\in \Z}\cup
\{(-\infty,k]\}_{k\in \Z},\ A\subseteq J\}$ are disjointly supported
unit vectors in $L_1$.\fin}
\end{remark}


\section{Edge Markov type need not imply Enflo type}\label{sec:edge}

A Markov chain $\{Z_t\}_{t=0}^\infty$ with transition probabilities
$a_{ij}\coloneqq\Pr(Z_{t+1}=j\mid Z_t=i)$ on the state space
$\{1,\ldots,n\}$ is {\em stationary\/} if $\pi_i\coloneqq\Pr(Z_t=i)$
does not depend on $t$ and it is {\em reversible\/} if
$\pi_i\,a_{ij}=\pi_j\,a_{ji}$ for every $i,j\in\{1,\ldots,n\}$.
Given a metric space $(X,d_X)$ and $p\in [1,\infty)$, we say that
$X$ has {\em Markov type} $p$ if there exists a constant $K>0$ such
that for every stationary reversible Markov chain
$\{Z_t\}_{t=0}^\infty$ on $\{1,\ldots,n\}$, every mapping
$f:\{1,\ldots,n\}\to X$ and every time $t\in \mathbb N$,
\begin{eqnarray}\label{eq:defMarkov}
\E \big[ d_X(f(Z_t),f(Z_0))^p\big]\le K^p\,t\,\E\big[
d_X(f(Z_1),f(Z_0))^p\big].
\end{eqnarray}
The least such $K$ is called the Markov type $p$ constant of $X$,
and is denoted $M_p(X)$. Similarly, given $D>0$ we let $M^{\le
D}_p(X)$ denote the least constant $K$
satisfying~\eqref{eq:defMarkov} with the additional restriction that
$d_X\left(f(Z_0),f(Z_1)\right)\le D$ holds pointwise. We call
$M^{\le D}_p(X)$ the $D$-bounded increment Markov type $p$ constant
of $X$. Finally, if $(X,d_X)$ is an unweighted graph equipped with
the shortest path metric then the {\em edge Markov type} $p$
constant of $X$, denoted $M_p^\edge(X)$, is the least constant $K$
satisfying~\eqref{eq:defMarkov} with the additional restriction that
$f(Z_0)f(Z_1)$ is an edge (pointwise).

The fact that $L_2$ has Markov type $2$ with constant $1$, first
noted by K. Ball~\cite{Bal}, follows from a simple spectral argument
(see also inequality (8) in~\cite{NPSS06}). Since for $p\in [1,2]$
the metric space $\left(L_p,\|x-y\|_2^{p/2}\right)$ embeds
isometrically into $L_2$ (see~\cite{WW75}), it follows that $L_p$
has Markov type $p$ with constant $1$. For $p>2$ it was shown
in~\cite{NPSS06} that $L_p$ has Markov type $2$ with constants
$O\left(\sqrt{p}\right)$. We refer to~\cite{NPSS06} for a
computation of the Markov type of various additional classes of
metric spaces.

A metric space $(X,d_X)$ is said to have {\em Enflo type $p$} if
there exists a constant $K$ such that for every $n\in \mathbb N$ and
every $f:\{-1,1\}^n\to X$,
\begin{multline}\label{eq:enflo type}
\E \left[d_X(f(\e),f(-\e))^p\right]\\\le T^p\sum_{j=1}^n \E \left[
d_X\left(f(\e_1,\ldots,\e_{j-1},\e_j,\e_{j+1},\ldots,\e_n),f(\e_1,\ldots,\e_{j-1},-\e_j,\e_{j+1},\ldots,\e_n)\right)^p\right],
\end{multline}
where the expectation is with respect to the uniform measure on
$\{-1,1\}^n$. In~\cite{NS02} it was shown that Markov type $p$
implies Enflo type $p$. We define analogously to the case of Markov
type the notions of bounded increment Enflo type and edge Enflo
type.

The notions of Enflo type and Markov type were introduced as
non-linear analogues of the fundamental Banach space notion of {\em
Rademacher type}. We refer
to~\cite{Enf76,BMW86,Bal,NS02,MN05,NPSS06} and the references
therein for background on this topic and many applications. In
Banach space theory the notion analogous to bounded increment Markov
type is known as {\em equal norm Rademacher type}. It is well known
(see~\cite{T-J89}) that for Banach spaces equal norm Rademacher type
$2$ implies Rademacher type $2$ and that for $1<p<2$ equal norm
Rademacher type $p$ implies Rademacher type $q$ for every $q<p$ (but
is {\em does not} generally imply Rademacher type $p$). It is
natural to ask whether the analogous phenomenon holds true for the
above metric analogues of Rademacher type. Here we show that this is
not the case.

It follows from Theorem~\ref{thm:hilbert} that $\alpha^*(\Z\bwr
\Z)\ge\frac23$. Therefore for every $0<\alpha<\frac23$ there is a
mapping $F:\Z\bwr \Z\to L_2$ such that
$$
x,y\in \Z\bwr \Z\implies d_{\Z\bwr \Z} (x,y)^\alpha\lesssim
\|F(x)-F(y)\|_2\lesssim d_{\Z\bwr \Z} (x,y).
$$
Fix a stationary reversible Markov chain $\{Z_t\}_{t=0}^\infty$ on
$\{1,\ldots,n\}$ and a mapping $f:\{1,\ldots,n\}\to \Z \bwr \Z$ such
that $d_{\Z\bwr \Z}\left(f(Z_0),f(Z_1)\right)\le D$ holds pointwise.
Using the fact that $L_2$ has Markov type $2$ with constant $1$ we
deduce that
\begin{multline*}
\E\left[d_{\Z\bwr
\Z}\big(f(Z_t),f(Z_0)\big)^{2\alpha}\right]\lesssim\E\left[\left\|F\circ
f(Z_t)-F\circ f(Z_0)\right\|_2^2\right]\le t\, \E\left[\left\|F\circ
f(Z_1)-F\circ f(Z_0)\right\|_2^2\right]\\\lesssim t\,
\E\left[d_{\Z\bwr \Z}\big(f(Z_1),f(Z_0)\big)^2\right]\lesssim
D^{2(1-\alpha)}t\, \E\left[d_{\Z\bwr
\Z}\big(f(Z_1),f(Z_0)\big)^{2\alpha}\right].
\end{multline*}
Thus
$$
M_{2\alpha}^{\le D} (\Z\bwr \Z)\lesssim D^{1-\alpha}.
$$
In particular $\Z\bwr \Z$ has $D$-bounded increment Markov type $p$
and edge Markov type $p$ for every $p<\frac43$.

On the other hand we claim that $\Z\bwr \Z$ does not have Enflo type
$p$ for any $p>1$. This is seen via an argument that was used by
Arzhantseva, Guba and Sapir in~\cite{AGS06}. Fix $n\in \N$ and
define $f:\{-1,1\}^n\to \Z\bwr \Z$ by
\begin{eqnarray}\label{eq:embed cube}
f(\e_1,\ldots,\e_n)\coloneqq
\left(\sum_{j=n+1}^{2n}\e_{j-n}n\delta_j,0\right),
\end{eqnarray}
where $\delta_j$ is the delta function supported at $j$. Then for
every $\e\in \{-1,1\}^n$,
\begin{eqnarray}\label{eq:diagonals}
d_{\Z\bwr \Z}\big(f(\e),f(-\e)\big)\approx
n^2
\end{eqnarray}
and for every $j\in \{1,\ldots, n\}$,
\begin{eqnarray}\label{eq:edges}d_{\Z\bwr
\Z}\left(f(\e_1,\ldots,\e_{j-1},\e_j,\e_{j+1},\ldots,\e_n),
f(\e_1,\ldots,\e_{j-1},-\e_j,\e_{j+1},\ldots,\e_n)\right)\approx n.
\end{eqnarray}
Therefore if $\Z\bwr \Z$ has Enflo type $p$, i.e. if~\eqref{eq:enflo
type} holds true, then for every $n\in \N$ we have $n^{2p}\lesssim
n^{p+1}$, implying that $p\le 1$.\qed

\section{A lower bound on $\beta^*(G\bwr H)$}\label{sec:beta}

In this section we shall prove~\eqref{eq:gromov}, which is a
generalization of \`Ershler's work~\cite{Ersh01}. Namely, we will
prove the following theorem:

\begin{thm}\label{thm:ershler} Let $G$ and $H$ be finitely generated
groups. If $H$ has linear growth (or equivalently, by Gromov's
theorem~\cite{Gromov81}, $H$ has a subgroup of finite index
isomorphic to $\Z$) then $\beta^*(G\bwr H)\ge
\frac{1+\beta^*(G)}{2}$. For all other finitely generated groups $H$
we have $\beta^*(G\bwr H)=1$.
\end{thm}

Assume that $G$ is generated by a finite symmetric set $S_G\subseteq
G$ and  $H$  is generated by a finite symmetric set $S_H\subseteq
H$. We also let $e_G, e_H$ denote the identity elements of $G$ and
$H$, respectively. Given $g_1,g_2\in G$ and $h\in H$ define a
mapping $f_{g_1,g_2}^h:H\to G$ by
$$
f_{g_1,g_2}^h(x)\coloneqq \left\{
\begin{array}{ll}g_1 &
\mathrm{if}\ x=e_H,\\
g_2 &\mathrm{if}\  x=h,\\
e_G &\mathrm{otherwise}.\end{array}\right.
$$
It is immediate to check that the set
$$S_{G\bwr H}\coloneqq
\left\{f_{g_1,g_2}^h:\ g_1,g_2\in S_G\ \mathrm{and}\ h\in
S_H\right\}
$$
is symmetric and generates $G\bwr H$.

{}From now on, we will assume that the metrics on $G$, $H$ and $G\bwr
H$ are induced by $S_G$, $S_H$ and $S_{G\bwr H}$, respectively.
Analogously we shall  denote by $\left\{W_k^G\right\}_{k=0}^\infty$,
$\left\{W_k^H\right\}_{k=0}^\infty$ and $\left\{W_k^{G\bwr
H}\right\}_{k=0}^\infty$ the corresponding random walks, starting at
the corresponding identity elements.

\begin{thm}\label{thm:yuval} Assume that for some $\beta\in [0,1]$
we have
\begin{eqnarray}\label{eq: G assumption}
\E \left[d_G\left(W_n^G,e_G\right)\right]\gtrsim n^\beta,
\end{eqnarray}
where the implied constant may depend on $S_G$. If $H$ has linear
growth then
\begin{eqnarray}\label{eq:linear growth}
\E \left[d_{G\bwr H}\left(W_n^{G\bwr H},e_{G\bwr
H}\right)\right]\gtrsim n^{\frac{1+\beta}{2}}.
\end{eqnarray}
If $H$ has quadratic growth then
\begin{eqnarray}\label{eq:quadratic growth}
\E \left[d_{G\bwr H}\left(W_n^{G\bwr H},e_{G\bwr
H}\right)\right]\gtrsim \frac{n}{(1+\log n)^{1-\beta}}.
\end{eqnarray}
If the random walk $\left\{W_n^{H}\right\}_{n=0}^\infty$ is
transient then
\begin{eqnarray}\label{eq:transient}
\E \left[d_{G\bwr H}\left(W_n^{G\bwr H},e_{G\bwr
H}\right)\right]\gtrsim n.
\end{eqnarray}
The implied constants in~\eqref{eq:linear growth},
\eqref{eq:quadratic growth} and~\eqref{eq:transient} may depend on
$S_G$ and $S_H$.
\end{thm}

Theorem~\ref{thm:ershler} is a consequence of
Theorem~\ref{thm:yuval} since by Varopoulos' celebrated
result~\cite{Var85,Var87} (which relies on Gromov's growth
theorem~\cite{Gromov81}. See~\cite{KV83} and~\cite{Woess00} for a
detailed discussion), the three possibilities in
Theorem~\ref{thm:yuval} are exhaustive for infinite finitely
generated groups $H$. In the case when the random walk on $H$ is
transient, Theorem~\ref{thm:yuval} was previously proved by
Ka{\u\i}manovich and Vershik in~\cite{KV83}.

The following lemma will be used in the proof of
Theorem~\ref{thm:yuval}.

\begin{lem}\label{lem:yuval} Define for $n\in \N$,
\begin{eqnarray*}\label{eq:varo}
\psi_H(n)\coloneqq \left\{\begin{array}{ll}\sqrt{n}& \textrm{if $H$
has linear growth,} \\
1+\log n & \textrm{if $H$ has quadratic growth,}\\
1 & \textrm{otherwise.}\end{array}\right.
\end{eqnarray*}
Then
\begin{eqnarray}\label{eq:return time}
\E\left[ \left|\left\{0\le k\le n:
W_k^H=e_H\right\}\right|^\beta\right]\gtrsim \psi_H(n)^\beta,
\end{eqnarray}
and
\begin{eqnarray}\label{eq:range}
\E \left[\left|W^H_{[0,n]}\right|\right]\gtrsim \frac{n}{\psi_H(n)},
\end{eqnarray}
where $W^H_{[0,n]}\coloneqq \left\{W_0^H,\ldots,W_n^H\right\}$.
\end{lem}

\begin{proof} By a theorem of
Varapoulos~\cite{Var85-cras,Var87} (see also~\cite{HSC93} and
Theorem 4.1 in~\cite{Woess00}) for every $k\ge 0$,
\begin{eqnarray}\label{eq:varop}
\Pr\left[W_k^H=e_H\right]+\Pr\left[W_{k+1}^H=e_H\right]\approx
\left\{\begin{array}{ll}\frac{1}{\sqrt {k+1}} & \textrm{if $H$ has
linear growth,}\\ \frac{1}{k+1}& \textrm{if $H$ has quadratic
growth,}\end{array}\right.
\end{eqnarray}
and if $H$ has super-quadratic growth then $\sum_{k=1}^\infty
\Pr\left[W_k^H=e_H\right]<\infty$. Hence, if we denote
$$
X_n\coloneqq \left|\left\{0\le k\le n:
W_k^H=e_H\right\}\right|=\sum_{k=0}^n \1_{\{W^H_k=e_H\}}
$$
then it follows that
\begin{eqnarray}\label{eq:first moment}
\E\left[X_n\right]=\sum_{k=0}^n \Pr
\left[W_k^H=e_H\right]\stackrel{\eqref{eq:varop}}{\approx}
\psi_H(n).
\end{eqnarray}
To prove~\eqref{eq:return time} note that
\begin{eqnarray*}\label{eq:second moment}
\E\left[X_n^2\right]=\sum_{i,j=0}^n \Pr \left[W_i^H=e_H\ \wedge\
W_j^H=e_H\right]\le 2\sum_{i=0}^n \sum_{k=0}^{n-i}
\Pr\left[W_i^H=e_H\right]\cdot \Pr\left[W_k^H=e_H\right]\le
2\left(\E\left[X_n\right]\right)^2\stackrel{\eqref{eq:first
moment}}{\approx} \psi_H(n)^2.
\end{eqnarray*}
Using H\"older's inequality we deduce that
$$
\psi_H(n)\approx
\E\left[X_n\right]=\E\left[X_n^{\frac{\beta}{2-\beta}}\cdot
X_n^{\frac{2-2\beta}{2-\beta}}\right]\le
\left(\E\left[X_n^\beta\right]\right)^{\frac{1}{2-\beta}}
\left(\E\left[X_n^2\right]\right)^{\frac{1-\beta}{2-\beta}}\lesssim
\left(\E\left[X_n^\beta\right]\right)^{\frac{1}{2-\beta}}\psi_H(n)^{\frac{2-2\beta}{2-\beta}}.
$$
This simplifies to $\E\left[X_n^\beta\right]\gtrsim
\psi_H(n)^\beta$, which is precisely~\eqref{eq:return time}.

We now pass to the proof of~\eqref{eq:range}. For every
$k\in\{1,\ldots,n\}$ denote by $V_1,\ldots,V_k$ the first $k$
elements of $H$ that were visited by the walk
$\left\{W_j^H\right\}_{j=0}^\infty$. Write
$$
Y_k\coloneqq \left|\left\{0\le j\le n: W_j^H\in
\left\{V_1,\ldots,V_k\right\}\right\}\right|.
$$
Then
$$
\E\left[Y_k\right]= \sum_{j=1}^k \E\left[ \left|\left\{0\le j\le n:
W_j^H=V_j\right\}\right|\right]\le k\sum_{r=0}^n \Pr
\left[W_r^H=e_H\right]\stackrel{\eqref{eq:varop}}{\approx}
k\psi_H(n).
$$
Therefore for every $k\in \N$,
$$
\Pr\left[\left|W^H_{[0,n]}\right|\le k\right]\le \Pr\left[Y_k\ge
n\right]\le \frac{\E\left[Y_k\right]}{n}\lesssim
\frac{k\psi_H(n)}{n}.
$$
Hence we can choose $k\approx \frac{n}{\psi_H(n)}$ for which $
\Pr\left[\left|W^H_{[0,n]}\right|\ge k\right]\ge \frac12$,
implying~\eqref{eq:range}.
\end{proof}

\begin{proof}[Proof of Theorem~\ref{thm:yuval}] We may assume that
$n\ge 4$. Let $Q_H: G\bwr H\to H$ be the natural projection, i.e.
$Q_H(f,x)\coloneqq x$. Also, for every $x\in H$ let $Q_G^x:G\bwr
H\to G$ be the projection $Q_G^x(f,y) \coloneqq f(x)$.

Fix $n\in \N$. For every $h\in H$ denote
$$
T_h\coloneqq \left|\left\{0\le k\le n:\ Q_H\left(W^{G\bwr
H}_k\right)=h\right\}\right|.
$$
The set of generators $S_{G\bwr H}$ was constructed so that the
random walk on $G\bwr H$ can be informally described as follows: at
each step the ``$H$ coordinate" is multiplied by a random element
$h\in S_H$. The ``$G$ coordinate" is multiplied by a random element
$g_1\in S_G$ at the original $H$ coordinate of the walker, and {\em
also} by a random element $g_2\in S_G$ (which is independent of
$g_1$) at the new $H$ coordinate of the walker. This immediately
implies that the projection $\left\{Q_H\left(W_k^{G\bwr
H}\right)\right\}_{k=0}^\infty$ has the same distribution as
$\left\{W_k^H\right\}_{k=0}^\infty$. Moreover, conditioned on
$\{T_h\}_{h\in H}$ and on $Q_H\left(W_{n}^{G\bwr H}\right)$, if
$h\in H\setminus \left\{e_H,Q_H\left(W_{n}^{G\bwr H}\right)\right\}$
then the element $Q_G^h\left(W_{n}^{G\bwr H}\right)\in G$ has the
same distribution as $W_{2T_h}^G$. If $h\in
\left\{e_H,Q_H\left(W_{n}^{G\bwr H}\right)\right\}$ and $e_H\neq
Q_H\left(W_{n}^{G\bwr H}\right)$ then $Q_G^h\left(W_{n}^{G\bwr
H}\right)$ has the same distribution as $W_{\max\{2T_h-1,0\}}^G$,
and if $e_H= Q_H\left(W_{n}^{G\bwr H}\right)$ then
$Q_G^h\left(W_{n}^{G\bwr H}\right)$ has the same distribution as
$W_{2T_h}^G$.

These observations imply, using~\eqref{eq: G assumption}, that for
every $h\in H$ we have $\E\left[d_{G}\left(Q^h_G\left(W_{n}^{G\bwr
H}\right),e_G\right)\right]\gtrsim \E \left[T_h^\beta\right]$.
Writing $A_\ell\coloneqq \left\{h=W_\ell^H\ \wedge\ h\notin
W_{[0,\ell-1]}^H\right\}$ we see that
$$
\E \left[T_h^\beta\right]\ge \sum_{\ell=0}^{\lfloor
n/2\rfloor}\Pr(A_\ell)\cdot\E\left[T_h^\beta\big|
A_\ell\right]\stackrel{\eqref{eq:return time}}{\ge}
\sum_{\ell=0}^{\lfloor
n/2\rfloor}\Pr(A_\ell)\cdot\psi_H(n/2)^\beta=\Pr\left[h\in
W^H_{[0,\lfloor n/2\rfloor]}\right]\psi_H(n/2)^\beta.
$$
Hence,
\begin{multline*}\label{eq:summands}
\E \left[d_{G\bwr H}\left(W_{n}^{G\bwr H},e_{G\bwr
H}\right)\right]\gtrsim \sum_{h\in H} \E \left[
d_G\left(Q_G^h\left(W_{n}^{G\bwr H}\right),e_G\right)\right]\gtrsim
\sum_{h\in H}\E \left[T_h^\beta\right]\\ \gtrsim
\psi_H\left(n\right)^\beta\sum_{h\in H}\Pr\left[h\in W^H_{[0,\lfloor
n/2\rfloor]}\right] = \psi_H(n)^\beta\cdot
\E\left[\left|W^H_{[0,\lfloor n/2\rfloor]}\right|\right]
\stackrel{\eqref{eq:range}}{\gtrsim} \frac{n}{\psi_H(n)^{1-\beta}}.
\end{multline*}
This is precisely the assertion of Theorem~\ref{thm:yuval}.
\end{proof}

\begin{remark}\label{rem:valette} {\em In~\cite{CSV07} de Cornulier,
Stalder and Valette show that if $G$ is a finite group then for
every $p\ge 1$ we have $\alpha^\#_p(G\bwr F_n)\ge \frac{1}{p}$,
where $F_n$ denotes the free group on $n\ge 2$ generators. Note that
in combination with Lemma~\ref{lem:p>2} this implies that we
actually  $\alpha^\#_p(G\bwr F_n)\ge
\max\left\{\frac{1}{p},\frac12\right\}$. This bound is sharp due to
Theorem~\ref{thm:equialpha} and the fact that $\beta^*(G\bwr
F_n)=1$.

In fact, we have the following stronger result: if $X$ is a Banach
space with modulus of smoothness of power type $p$, $G$ is a
nontrivial group, and  $H$ is a group whose volume growth is at
least quadratic, then $\alpha^*_X(G\bwr H)\le \frac{1}{p}$. In
particular $\alpha^*_p(G\bwr F_2)=
\max\left\{\frac{1}{p},\frac12\right\}$. To prove the above
assertion note that it is enough to deal with the case $G=C_2$. If
$H$ is amenable then by Theorem~\ref{thm:ershler} we have
$\beta^*(C_2\bwr H)=1$, so that the required result follows from the
result of~\cite{ANP07} and the fact that $X$ has Markov type
$p$~\cite{NPSS06}. If $H$ is nonamenable then it has exponential
growth (see~\cite{Pat88}). Thus $\gamma\coloneqq \lim_{r\to\infty}
|B(e_H,r)|^{1/r}>1$, where $B(x,r)$ denotes the ball of radius $r$
centered at $x$ in the word metric on $H$ (note that the existence
of the limit follows from submultiplicativity). Fix $\delta\in
(0,1)$  such that $\eta\coloneqq
\frac{(1-\delta)^2\gamma}{1+\delta}>1$ and let $k_0\in \N$ be such
that for all $k\ge k_0$ we have $[(1-\delta)\gamma]^k\le
|B(e_H,k)|\le [(1+\delta)\gamma]^k$. For $k\ge k_0$ let
$\{x_1,\ldots,x_N\}$ be a maximal subset of $B(e_H,2k)$ such that
the balls $\{B(x_i,k/2)\}_{i=1}^N$ are disjoint. Maximality implies
that the balls $\{B(x_i,k)\}_{i=1}^N$ cover $B(x,2k)$, so that
$$
[(1+\delta)\gamma]^kN\ge N|B(e_H,k)|\ge \left|\bigcup_{i=1}^N
B(x_i,k)\right|\ge |B(e_H,2k)|\ge [(1-\delta)\gamma]^{2k},
$$
which simplifies to give the lower bound $ N\ge \eta^k$. Thus $k \lesssim \log N$.

Fix $\alpha\in [0,1]$ and assume that $F:C_2\bwr H\to X$ satisfies
$$
x,y\in C_2\bwr H\implies d_{C_2\bwr H}(x,y)^\alpha\lesssim
\|F(x)-F(y)\|\lesssim d_{C_2\bwr H}(x,y).
$$
Our goal is to prove that $\alpha\le \frac{1}{p}$. For every
$\e=(\e_1,\ldots,\e_N)\in \{-1,1\}^N$ define $\psi_\e:H\to C_2$ by
$\psi_\e(x_i)=\frac{1+\e_i}{2}$, and $\psi_\e(x)=0$ if
$x\notin\{x_1,\ldots,x_N\}$. Let $f:\{-1,1\}^N\to C_2\bwr H$ be
given by $f(\e)=(f_\e,e_H)$. It is immediate to check that for all
$\e,\e'\in \{-1,1\}^N$ we have $\frac{k}{2}\|\e-\e'\|_1\le
\|f(\e)-f(\e')\|\le 4k\|\e-\e'\|_1.$ Metric spaces with Markov type
$p$ also have Enflo type $p$~\cite{NS02}, i.e. they
satisfy~\eqref{eq:enflo type}. Thus we can apply the Enflo type
inequality~\eqref{eq:enflo type} to the mapping $F\circ f:
\{-1,1\}^N\to X$ and deduce that $(Nk)^{\alpha p}\lesssim Nk^p$.
Consequently,
$   N^{\alpha p}  \lesssim N k^p \lesssim N (\log N)^p$.
Since the last inequality holds for arbitrarily large $N$, we infer that $\alpha p \le 1$.
 \fin }
\end{remark}

\section{Discussion and further questions}\label{sec:open}

In this section we discuss some natural questions that arise from
the results obtained in this paper. We start with the following
potential converse to~\eqref{eq:ANP}:

\begin{ques}\label{q:amen beta} Is it true that for every finitely
generated amenable group $G$,
\begin{eqnarray*}\label{eq:q beta}
\alpha^*(G)=\frac{1}{2\beta^*(G)}\quad ?
\end{eqnarray*}
\end{ques}
If true, Question~\ref{q:amen beta}, in combination with
Corollary~\ref{cor:iterated}, would imply a positive solution to the
following question:

\begin{ques}\label{q:wreath} Is it true that for every finitely
generated amenable group $G$,
$$
\alpha^*(G\bwr \Z)=\frac{2\alpha^*(G)}{2\alpha^*(G)+1}\quad ?
$$
\end{ques}
Additionally, since $\beta^*(G)\le 1$, a positive solution to
Question~\ref{q:amen beta} would imply a positive solution to the
following question:

\begin{ques}\label{q:1/2} Is it true that for every finitely
generated amenable group $G$,
$$
\alpha^*(G)\ge \frac12 \quad ?
$$
\end{ques}

Using~\eqref{eq:L1connection}, and arguing analogously to
Lemma~\ref{lem:lamp} while using the $L_1$ embedding of $C_2\bwr \Z$
in Section~\ref{sec:embed}, we have the following fact:
\begin{lem}\label{lem:L1} If a finitely generated group $G$ admits
a bi-Lipschitz embedding into $L_1$ then so does $G\bwr \Z$.
\end{lem}

\begin{ques}\label{q:L1}
Is it true that for every finitely generated amenable group $G$ we
have $\alpha^*_1(G)=1$?
\end{ques}

Since the metric space $\left(L_1,\sqrt{\|x-y\|_1}\right)$ embeds
isometrically into $L_2$ (see~\cite{WW75}), a positive solution to
Question~\ref{q:L1} would imply a positive solution to
Question~\ref{q:1/2}.

Our repertoire of groups $G$ for which we know the exact value of
$\alpha^*(G)$ is currently very limited. In particular, we do not
know the answer to the following question:

\begin{ques}\label{q:2/3}
Does there exist a finitely generated amenable group $G$ for which
$\alpha^*(G)$ is irrational? Does there exist a finitely generated
amenable group $G$ for which $\frac23<\alpha^*(G)<1$?
\end{ques}

In~\cite{Yu05} Yu proved that for every finitely generated
hyperbolic group $G$ there exists a large $p>2$ for which
$\alpha_p^\#(G) \ge \frac1{p}$. In view of
Theorem~\ref{thm:equialpha} it is natural to ask:

\begin{ques}\label{q:hyperbolic} Is it true that for every finitely
generated hyperbolic group $G$ there exists some $p\ge 1$ for which
$\alpha_p^\#(G)\ge \frac12$?
\end{ques}

We do not know the value of $\alpha_p^*(\Z\bwr \Z)$ for $1<p<2$. The
following lemma contains some bounds for this number:

\begin{lem}\label{lem:LpZwrZ} For every $1<p<2$,
\begin{eqnarray}\label{eq:LpZwrZ}
\frac{p}{2p-1}\le \alpha_p^*(\Z\bwr \Z)\le
\min\left\{\frac{p+1}{2p},\frac{4}{3p}\right\}.
\end{eqnarray}
\end{lem}

\begin{proof} The lower bound in~\eqref{eq:LpZwrZ} is an immediate
corollary of Theorem~\ref{thm:Lp}. Since $\beta^*(\Z\bwr \Z)\ge
\frac34$, the upper bound $\alpha_p^*(\Z\bwr \Z)\le\frac{4}{3p}$
follows immediately from the results of~\cite{ANP07} (or
alternatively Theorem~\ref{thm:equialpha}), using the fact that
$L_p$, $1<p<2$, has Markov type $p$. The remaining upper bound is an
application of the fact that $L_p$, $1<p<2$, has Enflo type $p$,
which is similar to an argument in~\cite{AGS06}. Indeed, fix a
mapping $F:\Z\bwr \Z\to L_p$ such that
$$
x,y\in \Z\bwr \Z\implies d_{\Z\bwr \Z} (x,y)^\alpha\lesssim
\|F(x)-F(y)\|_p\lesssim d_{\Z\bwr \Z} (x,y).
$$
Let $f:\{-1,1\}^n\to \Z\bwr \Z$ be as in~\eqref{eq:embed cube}.
Plugging the bounds in~\eqref{eq:diagonals} and~\eqref{eq:edges}
into the Enflo type $p$ inequality~\eqref{eq:enflo type} for the
mapping $F\circ f:\{-1,1\}^n \to L_p$, we see that for all $n\in \N$
we have $n^{2p\alpha}\lesssim n^{p+1}$, implying that $\alpha \le
\frac{p+1}{2p}$.
\end{proof}

\begin{ques}\label{q:ZbwrZ} Evaluate $\alpha_p^*(\Z\bwr \Z)$ for
$1<p<2$.
\end{ques}

We end with the following question which arises naturally from the
discussion in Section~\ref{sec:edge}:

\begin{ques}\label{q:edge} Does there exist a finitely generated
group $G$ which has edge Markov type $2$ but does not have Enflo
type $p$ for any $p>1$?
\end{ques}
We do not even know whether there exists a finitely generated group
$G$ which has edge Markov type $2$ but does not have Markov type
$2$. Note that the results of Section~\ref{sec:edge} imply that if
$1<p<\frac43$ then the metric space $\left(\Z\bwr \Z,d_{\Z\bwr
\Z}^{p/2}\right)$ has bounded increment Markov type $2$, but does
not have Enflo type $q$ for any $q>\frac{2}{p}$. However, this
metric is not a graph metric.

\section*{Acknowledgements} We are grateful to Laurent Saloff-Coste
for helpful comments.



\bibliographystyle{abbrv}
\bibliography{wreath-general}

\end{document}